\documentclass[11pt]{amsart}

\title[Genus two Heegaard splittings for non-prime 3-manifolds]
{Disk complexes and genus two Heegaard splittings for non-prime 3-manifolds}

\author{Sangbum Cho}
\thanks{The first-named author is supported by Basic Science Research Program through the National Research Foundation of Korea (NRF) funded by the Ministry of Education, Science and Technology (2012006520).}
\address{Department of Mathematics Education, Hanyang University, Seoul 133-791,
Korea}
\email{scho@hanyang.ac.kr}

\author{Yuya Koda}
\thanks{The second-named author is supported by
Japan Society for Promotion of Science (JSPS) Postdoctoral Fellowships for Research Abroad.}

\address{
Mathematical Institute \newline
\indent Tohoku University, Sendai, 980-8578, Japan \newline
\indent and \newline
\indent (Temporary) Dipartimento di Matematica  \newline
\indent Universit\`{a} di Pisa, Largo Bruno Pontecorvo 5, 56127 Pisa, Italy}
\email{koda@math.tohoku.ac.jp}

\usepackage{amsfonts,amsmath,amssymb,amscd}
\usepackage{amsthm}
\usepackage{latexsym}
\usepackage{graphicx}
\usepackage{psfrag}
\usepackage{color}

\theoremstyle{plain}
\newtheorem*{theorem*}{Theorem}
\newtheorem*{lemma*} {Lemma}
\newtheorem*{corollary*} {Corollary}
\newtheorem*{proposition*}{Proposition}
\newtheorem*{conjecture*}{Conjecture}
\newtheorem{theorem}{Theorem}[section]
\newtheorem{lemma}[theorem]{Lemma}

\theoremstyle{remark}

\theoremstyle{definition}

\newcommand{\Natural}{\mathbb{N}}
\newcommand{\Integer}{\mathbb{Z}}

\newcommand{\Rational}{\mathbb{Q}}



\begin{document}

\maketitle

\begin{abstract}
Given a genus two Heegaard splitting for a non-prime 3-manifold,
we define a special subcomplex of the disk complex for one of the handlebodies of the splitting,
and then show that it is contractible.
As applications, first we show that the complex of Haken spheres for the splitting is contractible,
which refines the results of Lei and Lei-Zhang.
Secondly, we classify all the genus two Heegaard splittings for non-prime 3-manifolds,
which is a generalization of the result of Montesinos-Safont.
Finally, we show that the mapping class group of the splitting, called the Goeritz group,
is finitely presented by giving its explicit presentation.
\end{abstract}

\vspace{1em}

\begin{small}
\hspace{2em}  \textbf{2010 Mathematics Subject Classification}:
57N10; 57M60


\hspace{2em}
\textbf{Keywords}:
mapping class group; Heegaard splitting; group action; tree
\end{small}

\section*{Introduction}

Every closed orientable 3-manifold $M$ can be decomposed into
two handlebodies $V$ and $W$ by cutting $M$ along a closed orientable surface $\Sigma$ embedded in it.
This is called a {\it Heegaard splitting} for the manifold $M$,
and denoted by the triple $(V, W; \Sigma)$.
The surface $\Sigma$ is called a {\it Heegaard surface}
and its genus is called the {\it genus} of the splitting.
A separating $2$-sphere $P$ in $M$ is called a
 {\it Haken sphere} for the splitting $(V, W; \Sigma)$
if $P$ intersects the Heegaard surface $\Sigma$ in a single essential circle.
If $(V, W; \Sigma)$ is a genus two Heegaard splitting for $M$
that admits a Haken sphere, then $M$ is one of the $3$-sphere,
$S^2 \times S^1$, lens spaces or their connected sums.
In particular, if the manifold $M$ is non-prime,
then $M$ is a connected sum whose summands are lens spaces or $S^2 \times S^1$.

In this paper, we study the genus two Heegaard splittings for non-prime $3$-manifolds.
Given a genus two Heegaard splitting $(V, W; \Sigma)$
for a closed orientable non-prime $3$-manifold $M$,
we define a special subcomplex of the disk complex for each of the handlebodies $V$ and $W$,
which we will call the semi-primitive disk complex, and then show that it is contractible.
The semi-primitive disk complex is an analogue of the primitive disk complexes studied in the authors' previous works \cite{Cho08, Cho12, CK12, CK13a, CK13b, Kod11} to find presentations of certain kinds of mapping class groups, including some Goeritz groups.

Understanding the structure of the semi-primitive disk complexes with their properties, we produce several applications.
First, we prove that the complex of Haken spheres is contractible for the genus two Heegaard splitting for any non-prime $3$-manifold.
The complex of Haken spheres is the simplicial complex whose vertices are isotopy classes of Haken spheres, and it has been an interesting problem to understand the structure of it since Scharlemann \cite{Sch04} showed that the complex for the genus two Heegaard splitting for the $3$-sphere is connected.
In Lei \cite{Lei05} and Lei-Zhang \cite{LZ04}, it was shown that the complexes of Haken spheres are connected for genus two Heegaard splittings for non-prime $3$-manifolds.
In Theorem \ref{cor:contractibility of sphere complexes} in this work, we refine their results in an alternative way, showing that those complexes are actually contractible.

Secondly, we classify all the genus two Heegaard splittings for non-prime $3$-manifolds.
Indeed, any non-prime $3$-manifold $M$ admits at most two different genus two Heegaard splittings, and it is known from Montesinos-Safont \cite{MS88} that,
if $M$ is the connected sum of two lens spaces $L(p, q_1)$ and $L(p, q_2)$,
then there exists a unique genus two Heegaard surface for $M$ up to homeomorphism
if and only if ${q_1}^2 \equiv 1$ or ${q_2}^2 \equiv 1 \pmod p$.
Including this result, we determine all the non-prime $3$-manifolds that admit unique Heegaard surfaces up to homeomorphism, which is stated in Theorem \ref{thm:Generalization of Montesinos-Safont theorem}.

The final application is to obtain a presentation of the mapping class group
of a genus two Heegaard splitting for a non-Haken $3$-manifold,
using the semi-primitive disk complex.
Such a group is called a (genus two) Goeritz group.
Precisely, the {\it Goeritz group} of a Heegaard splitting $(V, W; \Sigma)$
for a manifold $M$ is the group of
isotopy classes of orientation-preserving homeomorphisms of $M$ that preserve $V$ and $W$ setwise.
In Theorem \ref{thm:presentations of the Goeritz groups} in this work,
we show that the genus two Goeritz groups
for any non-prime $3$-manifolds are all finitely presented by giving their explicit presentations.

The Goeritz groups have been interesting objects in the study of Heegaard splittings.
For example, some interesting questions on Goeritz groups were proposed by Minsky in \cite{Go07}.
A Goeritz group will be ``small" when the gluing map of the two handlebodies
that defines the Heegaard splitting is sufficiently complicated.
Indeed, Namazi \cite{Nam07} showed that the Goeritz group is actually
a finite group when the Heegaard splitting has ``high'' {\it Hempel distance}.
Here, we just simply mention that the Hempel distance
is a measure of complexity of the gluing map that defines the splitting.
We refer to \cite{Hem01} for its precise definition.
Namazi's result is improved by Johnson in \cite{Joh10} showing
that the Goeritz group is finite if the Hempel distance of the splitting is at least four.
We refer the reader to \cite{Joh11, Joh12} for related topics.
The Goeritz groups of Heegaard splittings of low Hempel distance
are not as ``small" as in the case of the high Hempel distance.

For example, it is easy to see that the Goeritz group of
the genus $g$ Heegaard splitting for $\#_g(S^2 \times S^1)$,
which is the double of the genus $g$ handlebody $V$, is isomorphic to the mapping class group of $V$.
We note that the Hempel distance of this splitting is zero.
The mapping class group of a handlebody of genus at least two is, of course, not finite.
A finite generating set of this group is obtained by Suzuki \cite{Suz77} and its finite presentation is obtained by Grasse \cite{Gra89} and Wajnryb \cite{Waj98} independently.
See also \cite{McC91, Joh95}.

It is natural to ask if a given Goeritz group is finitely generated or presented, and so finding a generating set or a presentation of it has been an important problem.
But beyond the case of $\#_g(S^2 \times S^1)$, the generating sets or the presentations of the groups have been obtained only for few manifolds with their splittings of small genus.
In the case of the $3$-sphere, it is known that the Goeritz group for the genus two splitting is finitely presented from the works \cite{Goe33, Sch04, Akb08, Cho08}.
Further, a finite presentation of the Goeritz group of the genus two Heegaard splitting is obtained for each of the lens spaces $L(p, 1)$ in \cite{Cho12} and $S^2 \times S^1$ in \cite{CK13a}.
In addition, finite presentations of the genus two Goeritz groups of some other lens spaces are given in \cite{CK13b}.
For the higher genus Goeritz groups of the $3$-sphere and lens spaces, it is conjectured that they are all finitely presented but 
it is still an open problem.

This paper is organized as follows.
In Sections \ref{sec:Semi-primitive disks} and \ref{sec:The complex of semi-primitive disks},
we introduce semi-primitive disks with their various properties,
and then show that the semi-primitive disk complexes are contractible, by giving an explicit description of them.
In Section \ref{sec:The complex of Haken spheres}, the complex of Haken spheres are shown to be contractible
(Theorem \ref{cor:contractibility of sphere complexes}),
and in Section \ref{sec:Classification of genus two Heegaard splittings},
we give a classification of the genus two Heegaard splittings for non-prime $3$-manifolds (Theorem \ref{thm:Generalization of Montesinos-Safont theorem}).
In the final section, a finite presentation is given for the Goeritz group of each non-prime $3$-manifold with its genus two Heegaard splitting (Theorem \ref{thm:presentations of the Goeritz groups}).

\medskip\

By disks, pairs of disks, triples of disks properly embedded in a handlebody, we often mean
their isotopy classes throughout the paper.
Also, we often speak of Haken spheres of a Heegaard splitting
to mean their isotopy classes preserving the Heegaard splitting.
When we choose representatives of their isotopy classes,
we assume implicitly that
they intersect each other minimally and transversely.
Moreover, by homeomorphisms we often mean their isotopy classes when
it is obvious from context.

We use the standard notation of lens spaces as follows.
Let $V$ and $W$ be oriented solid tori.
Let $(m,l)$ be the pair of a meridian and a longitude of $V$.
We orient $m$ and $l$ in such a way that
the pair $(m,l)$ yields the orientation of
$\partial V$ induced by that of $V$.
The homology classes $[m]$ and $[l]$
of $m$ and $l$ induce a basis
of $H_1(\partial V)$.
In the same manner, we have the pair $(m', l')$ of a meridian and a longitude of $W$.
The lens space $L(p,q)$ is a 3-manifold obtained by identifying the boundaries of
$V$ and $W$ using an orientation-reversing
homeomorphism $\varphi : \partial V \to \partial W$ that
induces an isomorphism $\varphi_* : H_1(\partial V) \to H_1(\partial W)$
represented by $\left( \begin{array}{cc} q & p \\ s & -r \end{array} \right)$,
where $qr + ps = 1$.
In particular, $\varphi$ maps $m'$ to a $(p,q)$-curve with respect to $(m,l)$ on $\partial V$,
that is,
$\varphi_* [m'] = p [l] + q[m]$ in $H_1(\partial V)$.
We note that the image of $m$ by $\varphi^{-1}$ is a $(p, r)$-curve
with respect to $(m',l')$ on $\partial W$.
By definition, a lens space is equipped with a canonical orientation
induced from those of $V$ and $W$.
This orientation induces a canonical orientation of
the connected sum of two lens spaces.
Throughout the paper,
we will not regard $S^3 = L(1,0)$ nor $S^2 \times S^1 = L(0,1)$
as lens spaces.

\section{Semi-primitive disks}
\label{sec:Semi-primitive disks}

{An element of a free group $\mathbb Z \ast \mathbb Z$ of rank $2$ is said to be {\it primitive}
if it is a member of a generating pair of the group.
Primitive elements of $\mathbb Z \ast \mathbb Z$ have been well-understood.
For example, we refer the reader to \cite{OZ81}.
A key property of the primitive elements is that, fixing a generating pair $\{x, y\}$ of $\mathbb Z \ast \mathbb Z$, any primitive element has a cyclically reduced form which is a product of terms each of the form $x^\epsilon y^n$ and $x^\epsilon y^{n+1}$, or else a product of terms each of the form $y^\epsilon x^n$ and $y^\epsilon x^{n+1}$, for some $\epsilon \in \{1,-1\}$ and some $n \in \mathbb Z$.
The following is a direct consequence of this property.

\begin{lemma}
\label{lem:necessity condition to be a power of a primitive element}
Fix a generating pair $\{x , y\}$ of $\Integer * \Integer$.
Let $w$ be a cyclically reduced word on $\{x, y\}$.
If $w$ contains both $x$ and $x^{-1}$, both $y$ and $y^{-1}$ or
both $x^{\pm2}$ and $y^{\pm2}$ simultaneously,
then the element represented by $w$ is neither trivial nor a power of a primitive element.
\end{lemma}

Let $V$ be a genus two handlebody, and let $D$ and $E$ be
disjoint disks in $V$ such that $D \cup E$ cuts $V$ into a 3-ball.
We fix an orientation on each of $\partial D$ and $\partial E$, and then assign letters $x$ and $y$ to $\partial D$ and $\partial E$ respectively.
Let $l$ be an oriented simple closed curve on $\partial V$ which intersects
$\partial D \cup \partial E$ minimally and transversely.
Then $l$ determines a word on $\{x, y\}$ that can be read off
by the intersections of $l$ with $\partial D$ and $\partial E$.
We note that this word is well-defined up to cyclic conjugation.
The following is a simple criterion for triviality and primitiveness of the elements represented by $l$, which can be considered as a simpler version of Lemma 2.3 in \cite{CK12}.

\begin{lemma}
\label{lem:non-triviality and non-primitivity}
In the above setting, if a word $w$ determined by
the simple closed curve $l$ contains a subword
of the form $x y^p x^{-1}$ for some $p \in \Natural$, or $x^2 y^2$,
then any word determined by $l$ is cyclically reduced.
Moreover, the element represented by $w$ is neither trivial nor a power of a primitive element.
\end{lemma}

The idea of the proof is that, if $w$ contains one of those subwords, then any word determined by $l$ cannot contain $x^{\pm1}x^{\mp1}$ and $y^{\pm1}y^{\mp1}$, and any cyclically reduced word containing both $x$ and $x^{-1}$ or both $x^2$ and $y^2$ cannot represent a power of a primitive element by Lemma \ref{lem:necessity condition to be a power of a primitive element}.

Let $(V, W; \Sigma)$ be a genus two Heegaard splitting for a non-prime 3-manifold.
Recall that, by \cite{Hak68}, the splitting $(V, W; \Sigma)$ admits a Haken sphere.
A non-separating disk $D$ in $V$ is said to be {\it semi-primitive} if
there exists a Haken sphere $P$ of $(V, W; \Sigma)$ disjoint from $D$.
The next lemma follows from the definition.

\begin{lemma}
\label{lem:semi-primitive disks and words}
Let $(V, W; \Sigma)$ be a genus two Heegaard splitting for a non-prime $3$-manifold.
Let $D$ be a semi-primitive disk in $V$.
Then an element of $\pi_1(W)$ determined by $\partial D$ is either trivial or a power
of a primitive element.
\end{lemma}

We remark that there is a semi-primitive disk $D$ in $V$ such that $\partial D$ represents the trivial element of $\pi_1(W)$ if and only if the manifold has a $S^2 \times S^1$ summand.
In this case, $\partial D$ also bounds a disk in $W$.

\begin{lemma}
\label{lem:sufficiency and necessity condition to be semi-primitive}
Let $(V, W; \Sigma)$ be a genus two Heegaard splitting for a non-prime $3$-manifold.
Let $D$ be  a non-separating disk in $V$.
Then $D$ is semi-primitive if and only if
there exists a non-separating disk $E'$ in $W$ disjoint from
$D$.
\end{lemma}
\begin{proof}
The ``only if" part is trivial.
Let $E'$ be a non-separating disk in $W$ disjoint from $D$, and let $\Sigma'$ be the 4-holed sphere
obtained by cutting $\Sigma$ along $\partial D \cup \partial E'$.
Let $d^+$ and $d^-$ (${e'}^+$ and ${e'}^-$, respectively)
be the two boundary circles of $\Sigma'$ coming from
$\partial D$ ($\partial E'$, respectively).
Let $\alpha_{P}$ be an arbitrary simple arc in $\Sigma'$ connecting $d^+$ and $d^-$.
Then, up to isotopy,
there exists a unique simple arc $\alpha_{P}'$ in $\Sigma'$ connecting ${e'}^+$ and ${e'}^-$
such that $\alpha_{P} \cap \alpha_{P}' = \emptyset$.
We note that the frontier $\gamma_{P}$ of a regular neighborhood of $d^+ \cup \alpha_{P} \cup d^-$
coincides with the frontier of a regular neighborhood of
${e'}^+ \cup \alpha_{P}' \cup {e'}^-$ in $\Sigma'$.
It follows that $\gamma_{P}$ bounds a disk in each of
$V$ and $W$.
This implies that there exists a Haken sphere $P$ of $(V, W; \Sigma)$ such that
$P \cap \Sigma = \gamma_{P}$.
\end{proof}

In the proof above, every simple closed curve $\gamma_Q$
in $\Sigma'$ that separates $d^+ \cup d^-$ and $e'^+ \cup e'^-$
is the frontier of a regular neighborhood of the union of
$d^+ \cup d^-$ ($e'^+ \cup e'^-$, respectively) and a simple arc
$\alpha_Q$ ($\alpha_Q'$, respectively) in $\Sigma'$ connecting
$d^+$ and $d^-$ ($e'^+$ and $e'^-$, respectively).
Thus every essential, separating, simple closed curve
in $\Sigma$ disjoint from $\partial D \cup \partial E'$ bounds
separating disks in both $V$ and $W$.

\subsection{Connected sum of two lens spaces}
\label{subsec:Semi-primitive disks and a genus two Heegaard splitting for
the connected sum of two lens spaces}

Throughout this subsection, we always assume that
$(V, W; \Sigma)$ is a genus two Heegaard splitting for the
connected sum of two lens spaces.

\begin{lemma}
\label{lem:uniqueness of a disjoint semi-primitive disk}
Let $D$ be a semi-primitive disk in $V$.
Then there is a unique non-separating disk $E'$ in $W$ disjoint from
$D$.
\end{lemma}
\begin{proof}
By Lemma \ref{lem:sufficiency and necessity condition to be semi-primitive}, such a disk $E'$ exists.
To see the uniqueness, assume that there exist non-isotopic, non-separating disks
$E'_1$ and $E'_2$ in $W$ disjoint from $D$.
We assume that $E'_1$ and $E'_2$ intersect each other transversely and minimally.
If they have non-empty intersection, a disk obtained from $E'_1$ by a surgery along
an outermost subdisk of $E'_2$ cut off by $E'_1 \cap E'_2$ is also
a non-separating disks in $W$ disjoint from $D$.
This disk has fewer intersection with $E'_1$ than $E'_2$ had, and so by repeating surgeries if they still have intersection, we obtain a non-separating disk $E'$ in $W$ disjoint from $E'_1$ and from $D$.
Since $\partial D$ does not intersects $E'_1 \cup E'$, the circle $\partial D$ bounds
a disk $D'$ in $W$.
This implies that $D \cup D'$ is a non-separating sphere in
the connected sum of two lens spaces, whence a contradiction.
\end{proof}


The next theorem will play an important role in Section
\ref{sec:The complex of semi-primitive disks}.
\begin{theorem}
\label{thm:surgery of semi-primitive disks for lens and lens}
Let $D$ and $E$ be semi-primitive disks in $V$ that intersect each
other transversely and minimally.
Then at least one of the two disks obtained from $E$ by
a surgery along an outermost subdisk of $D$ cut off by $D \cap E$
is a semi-primitive disk.
\end{theorem}	
\begin{proof}
Let $C$ be an outermost subdisk of $D$ cut off by $D \cap E$.
Each Haken sphere $P$ of
$(V, W; \Sigma)$ disjoint from $E$ cuts the handlebody $V$ into two solid tori $V_1$ and $V_2$, and $W$ into $W_1$ and $W_2$.
We assume that
$E$ is the meridian disk of $V_1$, and that $V_1 \cup W_1$ and $V_2 \cup W_2$ are punctured lens spaces.
Let $E_0$, $E'$ and $E'_0$ be
the meridian disks of solid tori
 $V_2$, $W_1$ and $W_2$, respectively,
which are disjoint from $P$.
We choose a Haken sphere $P$ among all Haken spheres
disjoint from $E$ so that $|C \cap E_0|$ is minimal.
Assume that $\partial E'$ ($\partial E_0'$, respectively)
is a $(p_2, q_2)$-curve ($(p_1, q_1)$-curve, respectively) with respect to
the meridian $\partial E$ ($\partial E_0$, respectively)
and a fixed longitude on $\partial V_1$ ($\partial V_2$, respectively).
We may assume that $1 \leq q_1 < p_1$ and $1 \leq q_2 < p_2$. 
Each element of $\pi_1(W)$ can be represented by a word on
$\{ x, y \}$, where $x$ and $y$ are determined (up to sign) by
the meridian disks $E'$ and $E'_0$ respectively.
If $E_0$ is disjoint from $C$, then $E_0$ is one of the disks obtained from $E$ by a surgery along $C$,
and is a semi-primitive disk, so we are done.

Assume that $C \cap E_0 \neq \emptyset$.
Let $C_0$ be an outermost subdisk of $C$ cut off by $C \cap E_0$
such that
$C_0 \cap E = \emptyset$.
Let $\Sigma_0$ be the 4-holed sphere obtained by cutting $\Sigma$ along
$\partial E \cup \partial E_0$.
Let $e^+$ and $e^-$ (${e_0}^+$ and ${e_0}^-$, respectively)
be the boundary circles of
$\Sigma_0$ coming from $\partial E$ ($\partial E_0$, respectively).
Then $C_0 \cap \Sigma_0$ is the frontier of
a regular neighborhood of the union of one of
$e^+$ and $e^-$, say $e^+$, and a simple arc $\alpha_0$ connecting
$e^+$ and one of ${e_0}^+$ and ${e_0}^-$, say ${e_0}^+$.
Up to isotopy, the arc $\alpha_0$ does not intersect
$\partial E'_0$, otherwise a word of $\partial D$ would contain the subword $y x^{p_2} y^{-1}$ (after changing the orientations if necessary),
which contradicts
Lemmas \ref{lem:non-triviality and non-primitivity} and \ref{lem:semi-primitive disks and words}.
We denote by $E_1$ the disk obtained from $E_0$ by a surgery along $C_0$
that is not $E$.
We remark that $|C \cap E_1| < |C \cap E_0|$ and
that $\partial E_1$ determines a word of the form $x^{p_2} y^{p_1}$ (after changing the orientations if necessary).
See $\Sigma_0$ in Figure \ref{fig:sequence_of_disks}.
\begin{figure}[htbp]
\begin{center}
\includegraphics[width=14.5cm,clip]{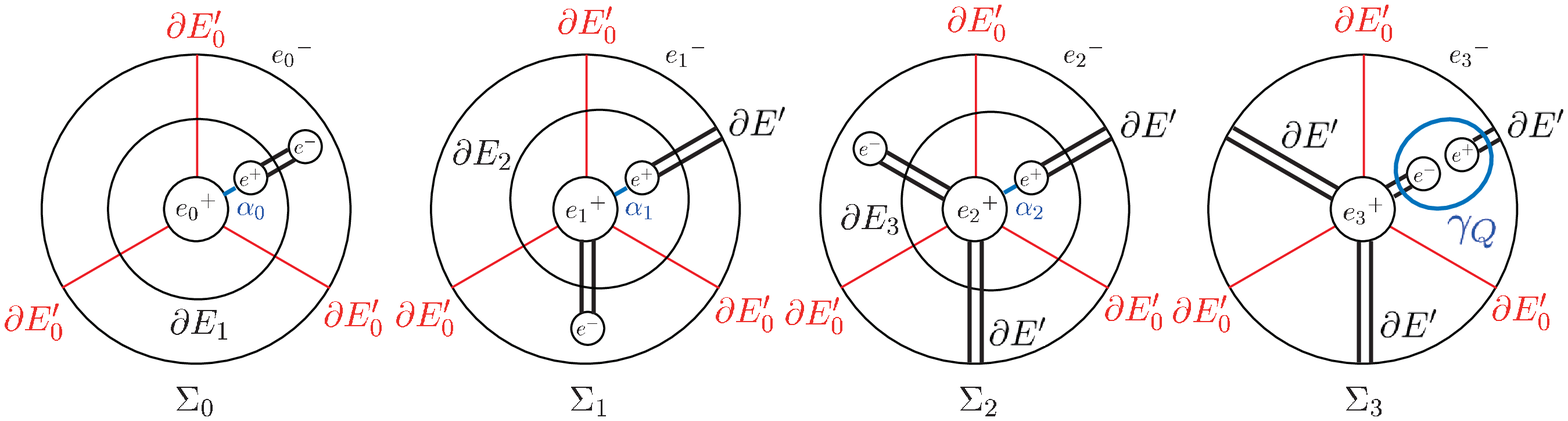}
\caption{The case where $(p_1, q_1) = (3,1)$ and $(p_2, q_2) = (2,1)$.
The circles $\partial E_1$, $\partial E_2$ and $\partial E_3$ determine the words
$x^2 y^3$, $x^2 y x^2 y^2$, $(x^2 y)^3$, respectively. }
\label{fig:sequence_of_disks}
\end{center}
\end{figure}

We define inductively a sequence of disks $E_2$, $E_3 , \ldots , E_{p_1}$ in $V$ as follows.
For $i \in \{1, 2, \ldots, p_1-1 \}$
let $\Sigma_i$ be the 4-holed sphere obtained by cutting $\Sigma$ along
$\partial E \cup \partial E_i$.
Let $e^+$ and $e^-$ (${e_i}^+$ and ${e_i}^-$, respectively)
be the boundary circles of
$\Sigma_i$ coming from $\partial E$ ($\partial E_i$, respectively).
Then there exists a unique simple arc $\alpha_i$ in $\Sigma_i$ connecting $e^+$ and one of ${e_i}^+$ or ${e_i}^-$ such that $\alpha_i$ is disjoint from $\partial E'_0$ and is not parallel to any arc component of $\partial E' \cap \Sigma_i$.
We may assume that $\alpha_i$ connects $e^+$ and ${e_i}^+$ by exchanging ${e_i}^+$ and ${e_i}^-$ if necessary.
Let $E_{i+1}$ be the disk obtained by the band sum of $E$ and $E_i$ along $\alpha_i$.
The disk $E_{i+1}$ is not isotopic to $E_{i-1}$ since the arc $\alpha_i$ is not parallel to any arc component of $\partial E' \cap \Sigma_i$.
See Figure \ref{fig:sequence_of_disks}.
We note that the circle $\partial E_2$ determines the word
$x^{p_2} y^{q_1} x^{p_2} y^{p_1 - q_1}$.
The circle $\partial E_3$ determines the word
$x^{p_2} y^{q_1} x^{p_2} y^{q_1} x^{p_2} y^{p_1 - 2 q_1}$ if $1 \leqslant q_1 \leqslant p_1 / 2$, and
$x^{p_2} y^{2p_1 - q_1} x^{p_2} y^{p_1 - q_1} x^{p_2} y^{p_1 - q_1}$ if $p_1 / 2 < q_1 < p_1 $.
Also, the circle $\partial E_{p_1-1}$ determines the word
$(x^{p_2} y)^{p_1 - q_1} y (x^{p_2} y)^{q_1 - 1}$. 
Finally, the circle $\partial E_{p_1}$ determines a word of the form
$(x^{p_2} y)^{p_1}$, which is apparently a power of a primitive element of
$\pi_1(W)$.

We show that $E_{p_1}$ is a semi-primitive disk and in fact there exists a Haken sphere disjoint from $E_{p_1}$ and $E$.
Let $\Sigma_{p_1}$ be the 4-holed sphere obtained by cutting $\Sigma$ along
$\partial E \cup \partial E_{p_1}$.
By the construction, the two boundary circles
$e^+$ and $e^-$ of $\Sigma_{p_1}$
coming from $\partial E$ are contained in the same component
of $\Sigma_{p_1}$ cut off by $\partial E'_0 \cap \Sigma_{p_1}$.
Hence there exists an arc $\alpha_Q$ in $\Sigma_{p_1}$ connecting
$e^+$ and $e^-$ such that $\alpha_Q \cap \partial E'_0 = \emptyset$.
We denote by $\gamma_Q$ the frontier of a regular neighborhood
of $e^+ \cup \alpha_Q \cup e^-$.
Apparently, $\gamma_Q$ is disjoint from
$E \cup E'_0$.
See the $4$-holed sphere $\Sigma_3$ in Figure \ref{fig:sequence_of_disks}. 
Thus it follows from the remark right after
Lemma \ref{lem:sufficiency and necessity condition to be semi-primitive}
that there exists a Haken sphere $Q$ in $(V, W, \Sigma)$
such that $Q \cap \Sigma = \gamma_Q$.
In particular, $Q$ is disjoint from $E_{p_1}$, and hence $E_{p_1}$ is a semi-primitive disk.

Now we claim that, for $i \in \{1, 2, \ldots, p_1-1 \}$,
$C \cap E_i \neq \emptyset$, and $E_{i+1}$ is obtained from $E_i$ by surgery along
an outermost subdisk $C_i$ of $C$ cut off by $C \cap E_i$ such that
$C_i \cap E = \emptyset$.
The latter claim follows immediately from the former one, since, if $C$ intersects $E_i$,
then $C_i \cap \Sigma_i$ is the frontier of
a regular neighborhood of $e^+ \cup \alpha_i$ in $\Sigma_i$, and so the same reason to the case of $\alpha_0$ implies the latter claim.
Suppose that $E_i$ is the first disk disjoint from $C$ for contradiction.

First, assume that $i \in \{ 1 , 2 , \ldots, p_1 -2 \}$.
Since $C$ does not intersect $E_i$, the intersection $C \cap \Sigma_i$ is a simple arc with both end points on
$e^{\epsilon_1}$, where $\epsilon_1 \in \{+ , -\}$.
Then $C \cap \Sigma_i$ is the frontier
of a regular neighborhood of ${e_i}^{\epsilon_1} \cup \beta_{\epsilon_1 \epsilon_2}$, where
$\epsilon_2 \in \{+ , -\}$ and $\beta_{\epsilon_1 \epsilon_2}$ is a simple arc in
$\Sigma_{i}$ connecting $e^{\epsilon_1}$ and $e_{i}^{\epsilon_2}$.
We see that $\beta_{\epsilon_1 \epsilon_2}$ is disjoint from $\partial E'_0 \cap \Sigma_i$, otherwise
$C \cap \Sigma_i$ would give a word containing $y x^{p_2} y^{-1}$ and hence $D$ is not a semi-primitive disk by
Lemmas \ref{lem:non-triviality and non-primitivity} and \ref{lem:semi-primitive disks and words}, a contradiction.
If $\epsilon_1 \neq \epsilon_2$, then we may isotope $C \cap \Sigma_i$ on  $\Sigma_{i}$ so that
$C \cap \Sigma_i$ is disjoint from $E_{i-1}$.
See Figure \ref{fig:the_disk_c}.
This contradicts the assumption that $C$ intersects $E_{i-1}$.
Thus we have $\epsilon_1 = \epsilon_2$.
We assumed that $i \leqslant p_1 - 2$, and hence there exists at least one arc component of $C$ cut off by $\partial E'_0$ that does not
intersect $\partial E'$, which means a word determined by $C \cap \Sigma_i$ contains $y^2$.
Therefore $C \cap \Sigma_i$ gives a word containing $x^{p_2}y^2$, and so containing $x^2y^2$.
\begin{figure}[htbp]
\begin{center}
\includegraphics[width=10cm,clip]{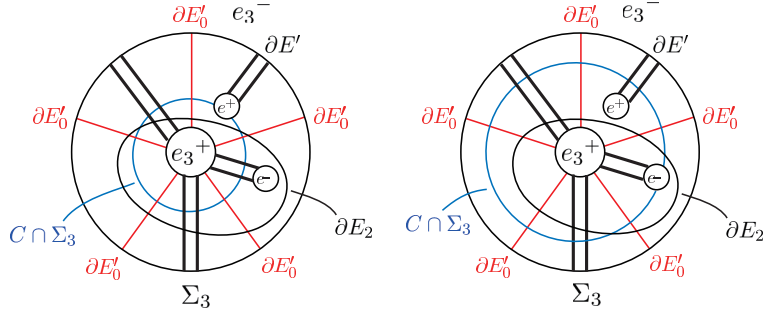}
\caption{The case where $(p_1, q_1) = (5,2)$, $(p_2, q_2) = (2,1)$ and $i=3$.}
\label{fig:the_disk_c}
\end{center}
\end{figure}
Again, this implies that $D$ is not a semi-primitive disk 
by Lemmas \ref{lem:non-triviality and non-primitivity} and \ref{lem:semi-primitive disks and words},
whence a contradiction.
(We note that, when $i = p_1 - 1$, the word determined by $C \cap \Sigma_{p_1 - 1}$ is of the form $yx^{p_2}yx^{p_2} \cdots yx^{p_2}y$, and so it does not contain $y^2$.)
Next, assume that $i=  p_1 -1$.
In this case, $C$ is disjoint from $E_{ p_1 -1}$ and intersects $E_{ p_1 -2}$.
Then one of the resulting disks obtained by surgery on
$E$ along $C$ is $E_{ p_1 -1}$, and the other one is the semi-primitive disk $E_{ p_1 }$.
In particular, $C$ is disjoint from $E_{p_1}$.
This contradicts the minimality of $|C \cap E_0|$ since we are assuming that
$C \cap E_0 \neq \emptyset$.
Hence we get the claim.

However, this is impossible since
now we have
the inequalities
$|C \cap E_{p_1}| < |C \cap E_{p_1 - 1}| < \cdots < |C \cap E_0|$
and this contradicts, again, the minimality of $|C \cap E_0|$ .
\end{proof}	

\begin{lemma}
\label{lem:disjoint semi-primitive disks and reducing spheres}
Let $D$ and $E$ be disjoint, non-isotopic semi-primitive disks in $V$.
Then there exists a unique Haken sphere of $(V, W; \Sigma)$
disjoint from $D \cup E$.
\end{lemma}	
\begin{proof}
The uniqueness follows immediately from
Lemma \ref{lem:uniqueness of a disjoint semi-primitive disk}.
To show the existence of
a Haken sphere of $(V, W; \Sigma)$
disjoint from $D \cup E$,
we choose a Haken sphere $P$ among all Haken spheres
disjoint from $E$ so that $|D \cap E_0|$ is minimal as in the proof of
Theorem \ref{thm:surgery of semi-primitive disks for lens and lens}.
Also, we take the disks $E'$ and $E'_0$ in $W$ as in the
proof of Theorem \ref{thm:surgery of semi-primitive disks for lens and lens}.
Each element of $\pi_1(W)$ are represented by a word on
$\{ x, y \}$, where $x$ and $y$ are determined (up to sign) by
the meridian disks $E'$ and $E'_0$.
If $D = E_0$, we are done.
Assume that $D \neq E_0$ and $D \cap E_0 = \emptyset$.
Then the disk $D$ is the band sum of $E$ and $E_0$ along an arc, say $\alpha_0$, which connects $\partial E$ and $\partial E_0$. Since we assumed that $D$ is semi-primitive, the arc $\alpha_0$ is disjoint from $E'_0$ by the same reason to the case of the arc $\alpha_0$ in the proof of Theorem \ref{thm:surgery of semi-primitive disks for lens and lens} (after changing the orientations if necessary). Considering $\partial D$ as a circle lying in the $4$-holed sphere $\Sigma$ cut off by $\partial E \cup \partial E_0$, which is the same case to the circle $\partial E_1$ in $\Sigma_0$ in Theorem \ref{thm:surgery of semi-primitive disks for lens and lens}, we observe that a word determined by $\partial D$ must contain
a subword of the form $x^2 y^2$.
By Lemmas \ref{lem:non-triviality and non-primitivity}
and \ref{lem:semi-primitive disks and words}, the disk $D$ cannot be semi-primitive, a contradiction.
Finally, assume that $D \cap E_0 \neq \emptyset$.
Then by the same argument as the proof of
Theorem \ref{thm:surgery of semi-primitive disks for lens and lens}
for the disk $D$ instead of the outermost subdisk $C$,
we can deduce a contradiction.
\end{proof}	

\begin{lemma}
\label{lem:semi-primitive disk complex is one-dimensional}
Let $D$, $E$ and $F$ be pairwise disjoint, pairwise non-isotopic, non-separating disks in $V$.
If $D$ and $E$ are semi-primitive disks, then
$F$ is not a semi-primitive disk.
\end{lemma}	
\begin{proof}
By Lemma \ref{lem:disjoint semi-primitive disks and reducing spheres},
there exists a (unique) Heken sphere $P$ of $(V, W; \Sigma)$
disjoint from $D \cup E$. Thus we have the meridian disks $D'$ and $E'$ of the two solid tori
$W$ cut off by $P \cap W$ that are disjoint from $P$. 
Then the non-separating disk $F$ is the band sum of $D$ and $E$ along an arc, say $\alpha_0$, which connects $\partial D$ and $\partial E$.
This is exactly the case of ``$D \neq E_0$ and $D \cap E_0 = \emptyset$'' in the proof of Lemma \ref{lem:disjoint semi-primitive disks and reducing spheres}.
Here $D$, $E$, $F$, $D'$ and $E'$ correspond to $E_0$, $E$, $D$, $E'_0$ and $E'$, respectively, in the proof of Lemma \ref{lem:disjoint semi-primitive disks and reducing spheres}.
Thus, by the same reasoning, we see that $F$ is not semi-primitive.
\end{proof}

\subsection{Connected sum of $S^2 \times S^1$ and a lens space}
\label{subsec:Semi-primitive disks and a genus two Heegaard splitting for
the connected sum of a lens space and S2 times S1}

Throughout this subsection, we always assume that
$(V, W; \Sigma)$ is a genus two Heegaard splitting for the
connected sum of $S^2 \times S^1$ and a lens space.
A non-separating disk $D$ in $V$ is called a {\it reducing disk} if
$\partial D$ bounds a disk in $W$.
We remark that a reducing disk is also a semi-primitive one 
and the boundary circle of a reducing disk represents the trivial element of $\pi_1(W)$.

\begin{lemma}
\label{lem:the uniqueness of the reducing disk}
Let $D$ be a reducing disk in $V$.
Let $E$ be a non-separating disk in $V$ that is not isotopic to $D$.
\begin{enumerate}
\item
If $E$ is disjoint from $D$, then there exists a Haken sphere of $(V, W; \Sigma)$
disjoint from $D \cup E$.
In particular, $E$ is a semi-primitive disk but is not a reducing disk.
\item
If $E$ intersects $D$, then
$E$ is not a semi-primitive disk.
\end{enumerate}
\end{lemma}
\begin{proof}
\noindent (1)
Let $\Sigma'$ be the 4-holed sphere
obtained by cutting $\Sigma$ along $\partial D \cup \partial E$.
Let $d^+$ and $d^-$
be the two boundary circles of $\Sigma'$ coming from
$\partial D$.
Let $\alpha_P$ be an arbitrary simple arc in $\Sigma'$ connecting $d^+$ and $d^-$.
Since $D$ is a reducing disk,
the frontier $\gamma_P$ of a regular neighborhood of $d^+ \cup \alpha_P \cup d^-$
bounds a disk in each of $V$ and $W$.
This implies that there exists a Haken sphere $P$ of $(V, W; \Sigma)$ such that
$P \cap \Sigma = \gamma_P$, which is disjoint from $D \cup E$.

\noindent (2)
Let $D'$ be a disk in $W$ bounded by $\partial D$.
Let $C$ be an outermost subdisk of $E$ cut off by $D \cap E$.
Then a standard cut-and-paste argument allows us to have
a non-separating disk $E_1$ in $V$ that is not isotopic to $D$ and
disjoint from $C \cup D$.
By (1), $E_1$ is a semi-primitive disk.
Let $P$ be a Haken sphere of $(V, W; \Sigma)$ disjoint from $D \cup E_1$.
Let $E_1'$ be the semi-primitive disk in $W$ disjoint from $P$ that is not isotopic to $D'$.
Let $\Sigma'$ be the 4-holed sphere obtained by
cutting $\Sigma$ along $\partial D \cup \partial E_1$.
Let $d^+$ and $d^-$ (${e_1}^+$ and ${e_1}^-$, respectively)
be the two boundary circles of $\Sigma'$ coming from
$\partial D$ ($\partial E_1$, respectively).
We note that $\partial E_1' \cap \Sigma'$ cuts $\Sigma'$ into a finite number of rectangles and
a single rectangle with two holes $d^+$ and $d^-$.
See Figure \ref{fig:band_sum} (i).
\begin{figure}[htbp]
\begin{center}
\includegraphics[width=11cm,clip]{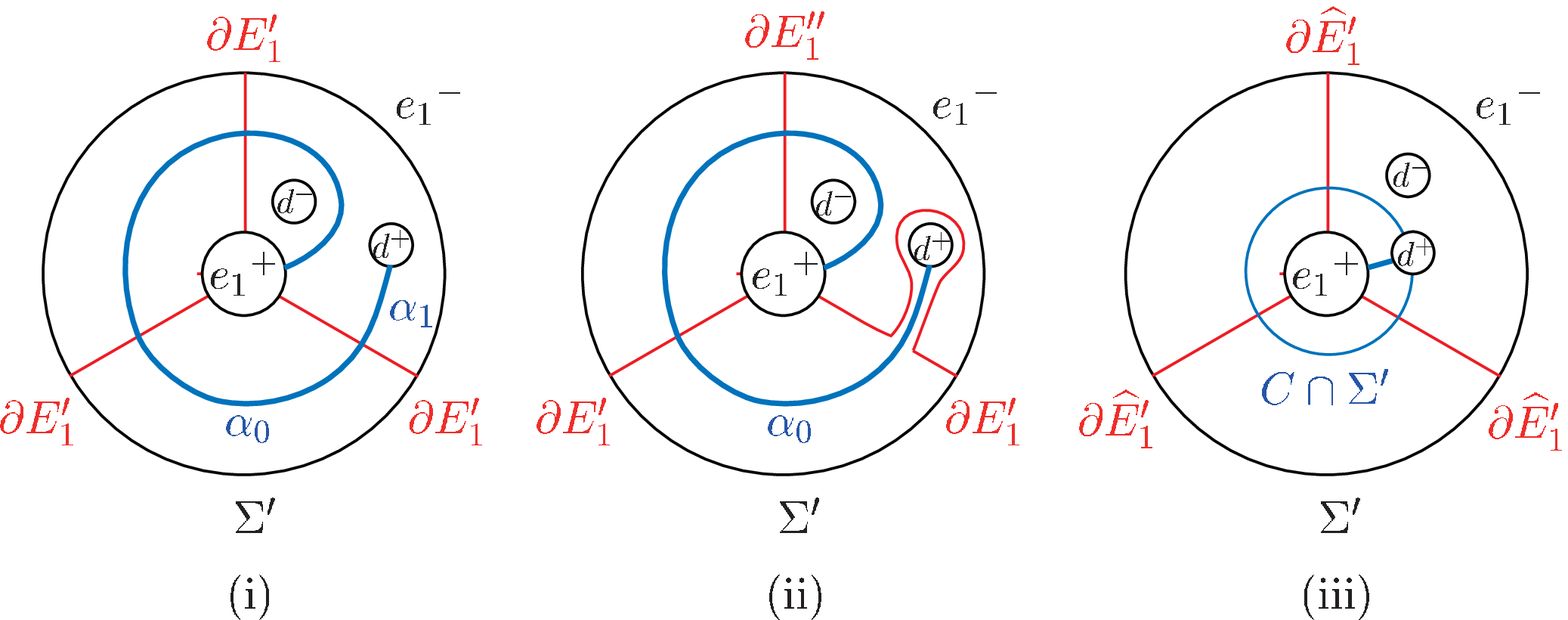}
\caption{}
\label{fig:band_sum}
\end{center}
\end{figure}
Then $C \cap \Sigma'$ is the frontier of a regular neighborhood of the union
of an arc $\alpha_0$ in $\Sigma'$ connecting one of $d^+$ and $d^-$, say $d^+$ and one of
$e_1^+$ and $e_1^-$, say $e_1^+$, and the boundary circle $e_1^+$.

Assume that $\alpha_0$ meets $\partial E_1'$.
Let $\alpha_1$ be a subarc of $\alpha_0$ connecting $d^+$ and $\partial E_1'$ such that
the interior of $\alpha_1$ is disjoint from $\partial E_1'$.
Let $E_1'' \subset W$ be the band sum of $E_1'$ and $D'$ along $\alpha_1$.
$E_1''$ is a semi-primitive disk and we have $|\alpha_0 \cap \partial E_1''| < |\alpha_0 \cap \partial E_1'|$.
See Figure \ref{fig:band_sum} (ii).
Repeating this process finitely many times, we obtain a semi-primitive disk $\widehat{E}_1'$ in $W$
disjoint from both $D$ and $\alpha_0$.

We give letters $x$ and $y$ to the circles $\partial D'$ and $\partial \widehat{E}_1'$, respectively,
after fixing an orientation of each of them.
Then a word on $\{ x ,y \}$ determined by $\partial E$ contains
a subword of the form $x y^p x^{-1}$,
which is determined by the subarc $C \cap \Sigma'$ after changing the orientations
if necessary.
See Figure \ref{fig:band_sum} (iii).
by Lemma \ref{lem:non-triviality and non-primitivity}, $E$ is neither
a reducing disk nor a semi-primitive disk.
\end{proof}

By Lemma \ref{lem:the uniqueness of the reducing disk},
$(V, W; \Sigma)$ admits a unique reducing disk.
The next lemma follows immediately from
the definition of a reducing disk and
the proof of Lemma \ref{lem:sufficiency and necessity condition to be semi-primitive}.
\begin{lemma}
\label{lem:disjoint}
Let $D$ be the reducing disk in $V$.
Then any non-reducing, semi-primitive disk in $V$ is disjoint from $D$ up to isotopy.
\end{lemma}


\section{The complex of semi-primitive disks}
\label{sec:The complex of semi-primitive disks}

Let $V$ be a handlebody.
The {\it disk complex} $\mathcal{K}(V)$ of $V$
is the simplicial complex whose vertices are the isotopy classes of essential disks in $V$ such that
the collection of distinct $k+1$ vertices spans a
$k$-simplex if they admit a
set of pairwise disjoint representatives.
The full-subcomplex $\mathcal{D}(V)$ of $\mathcal{K}(D)$
spanned by the vertices corresponding to non-separating disks is called the
{\it non-separating disk complex} of $V$.
In \cite{McC91}, it is shown that both
$\mathcal{K}(V)$ and $\mathcal{D}(V)$ are contractible.
Moreover, we have the following theorem.
\begin{theorem}[\cite{McC91, Cho08}]
\label{thm:surgery and contractibility}
A full subcomplex $\mathcal{L}$ of the disk complex $\mathcal{K}(V)$
is contractible if, given any two representative disks $E$ and $D$ of vertices of $\mathcal{L}$
intersecting each other transversely and minimally,
at least one of the
disks from surgery on $E$ along an outermost subdisk of $D$ cut off by
$D \cap E$ represents a vertex of $\mathcal{L}$.
\end{theorem}
Let $M_1$ be a lens space or $S^2 \times S^1$, and let $M_2$ be a lens space.
Let $(V, W; \Sigma)$ be a genus two Heegaard splitting for $M_1 \# M_2$.
The {\it semi-primitive disk complex} $\mathcal{SP}(V)$ of $V$
is the full subcomplex of $\mathcal{D}(V)$ spanned by the vertices corresponding to
semi-primitive disks of $V$.
We remark that the Goeritz group $\mathcal{G}$ of $(V, W; \Sigma)$ acts
on $\mathcal{SP}(V)$ simplicially.

\begin{theorem}
\label{thm:contractibility for lens and lens}
Let $M_1$ be a lens space or $S^2 \times S^1$, and let $M_2$ be a lens space.
Let $(V, W; \Sigma)$ be a genus two Heegaard splitting for $M_1 \# M_2$.
\begin{enumerate}
\item
If $M_1$ is a lens space,
then $\mathcal{SP}(V)$ is a tree.
\item
If $M_1 = S^2 \times S^1$,
then $\mathcal{SP}(V)$ is the cone of a tree.
\end{enumerate}
\end{theorem}
\begin{proof}
\noindent (1)  That $\mathcal{SP}(V)$ is contractible is a straightforward consequence of
Theorems \ref{thm:surgery of semi-primitive disks for lens and lens} and
\ref{thm:surgery and contractibility}.  That it is a 1-complex follows from Lemma \ref{lem:semi-primitive disk complex is one-dimensional}. \\
(2)
Let $D$ be the unique reducing disk in $V$.
Let $\mathcal{SP}_D(V)$ denote the full subcomplex of
$\mathcal{D}(V)$ spanned by the vertices corresponding to non-reducing semi-primitive disks.
By Lemmas \ref{lem:the uniqueness of the reducing disk} and \ref{lem:disjoint},
the complex $\mathcal{SP}_D(V)$ is the link of the vertex corresponding to $D$ in $\mathcal{D}(V)$.
It is shown in \cite{Cho08,McC91} that
the link of any vertex of $\mathcal{D}(V)$ is a tree, and hence $\mathcal{SP}_D(V)$ is a tree.
\end{proof}

\section{The complex of Haken spheres}
\label{sec:The complex of Haken spheres}
Let $(V, W; \Sigma)$ be a genus two Heegaard splitting for a closed orientable 3-manifold $M$.
The {\it complex $\mathcal{H}$ of Haken spheres}
of the splitting $(V, W; \Sigma)$
is defined to be the simplicial complex whose vertices
consists of the isotopy classes of Haken spheres such that the collection $P_0, P_1, \ldots, P_k$  of distinct $k+1$ vertices spans a
$k$-simplex if
$|P_i \cap \Sigma \cap P_j| = 4$ for all $0 \leqslant i < j \leqslant k$.
It is shown that the complex of Haken spheres of the genus two splitting for $S^3$ is connected by Scharlemann \cite{Sch04}, and it turns out that the complex actually deformation retracts to a tree from the works \cite{Akb08} and \cite{Cho08}.
Lei \cite{Lei05} and Lei-Zhang \cite{LZ04} showed
that the complex of Haken spheres of the genus two splitting for a non-prime $3$-manifold is connected.
In this section, we refine the results of Lei and Lei-Zhang.
That is, we show that the complexes of Haken spheres for non-prime $3$-manifolds are connected in a new way, and further show that they are actually contractible.
We use the results on the semi-primitive disk complexes developed in the previous section.

\begin{theorem}
\label{cor:contractibility of sphere complexes}
Let $M_1$ be a lens space or $S^2 \times S^1$, and let $M_2$ be a lens space.
Let $(V, W; \Sigma)$ be a genus two Heegaard splitting for $M_1 \# M_2$.
Then the complex $\mathcal{H}$ of Haken spheres of the splitting $(V, W; \Sigma)$ is
contractible.
The dimension of $\mathcal{H}$ is $1$, that is, $\mathcal{H}$ is a tree, if $M_1$ is a lens space,
and is $3$ if $M_1$ is $S^2 \times S^1$.
\end{theorem}
\begin{proof}
Let us assume first that $M_1$ is a lens space.
In Theorem \ref{thm:contractibility for lens and lens}, we have seen that
the semi-primitive disk $\mathcal{SP}(V)$ is a tree.
Let $\mathcal{SP}'(V)$ be the first barycentric
subdivision of the tree $\mathcal{SP}(V)$.
The tree $\mathcal{SP}'(V)$ is bipartite,  of which we call the vertices of countably infinite valence
(the vertices of the original $\mathcal{SP}(V)$) the black
vertices, and the vertices of valence 2 the white ones.
By Lemma \ref{lem:disjoint semi-primitive disks and reducing spheres},
the set of the white vertices one-to-one
corresponds to the set of Haken spheres.

Let $D$ be a semi-primitive disk in $V$.
We note that $D$ represents a black vertex of the tree $\mathcal{SP}'(V)$.
By Lemma \ref{lem:uniqueness of a disjoint semi-primitive disk},
there exists the unique semi-primitive disk $E'$ in $W$ disjoint from $D$.
The set of white vertices in the link of $D$ in $\mathcal{SP}'(V)$
one-to-one correspond of the set of the Haken spheres disjoint from $D \cup E'$.
Let $\Sigma'$ be the 4-holed sphere obtained by cutting $\Sigma$ off along
$\partial D \cup \partial E'$.
Let $d^+$ and $d^-$ (${e'}^+$ and ${e'}^-$, respectively)
be the two boundary circles of $\Sigma'$ coming from
$\partial D$ ($\partial E'$, respectively).
Let $\mathcal{H}_{D}$ be the full subcomplex of the complex $\mathcal{H}$
spanned by the vertices corresponding to Haken spheres disjoint from $D$.
We assign each vertex of $\mathcal{H}_D$ an element of
$\Rational_{ \mathrm{odd}} \cup \{ \infty \}$ in the following way.
Fix a Haken sphere $P$ of $(V, W; \Sigma)$ disjoint from $D \cup E'$.
Set $\mu = P \cap \Sigma'$ and fix a separating simple closed curve $\nu$ in $\Sigma'$ such that
$\nu$ separates $d^+ \cup {e'}^+$ and $d^- \cup {e'}^-$, and that
$| \mu \cap \nu | = 2$ after minimizing the intersection.
Let $\tilde{\Sigma'}$ be the covering space of $\Sigma'$ such that
\begin{itemize}
\item
the components of the preimage of $\mu$ ($\nu$, respectively) are
the vertical (horizontal, respectively) lines in the Euclidean plane;
\item
the set of components of the preimage of $\partial D$ correspond to
the set of points whose coordinates consist of integers.
\end{itemize}
\begin{figure}[htbp]
\begin{center}
\includegraphics[width=11.5cm,clip]{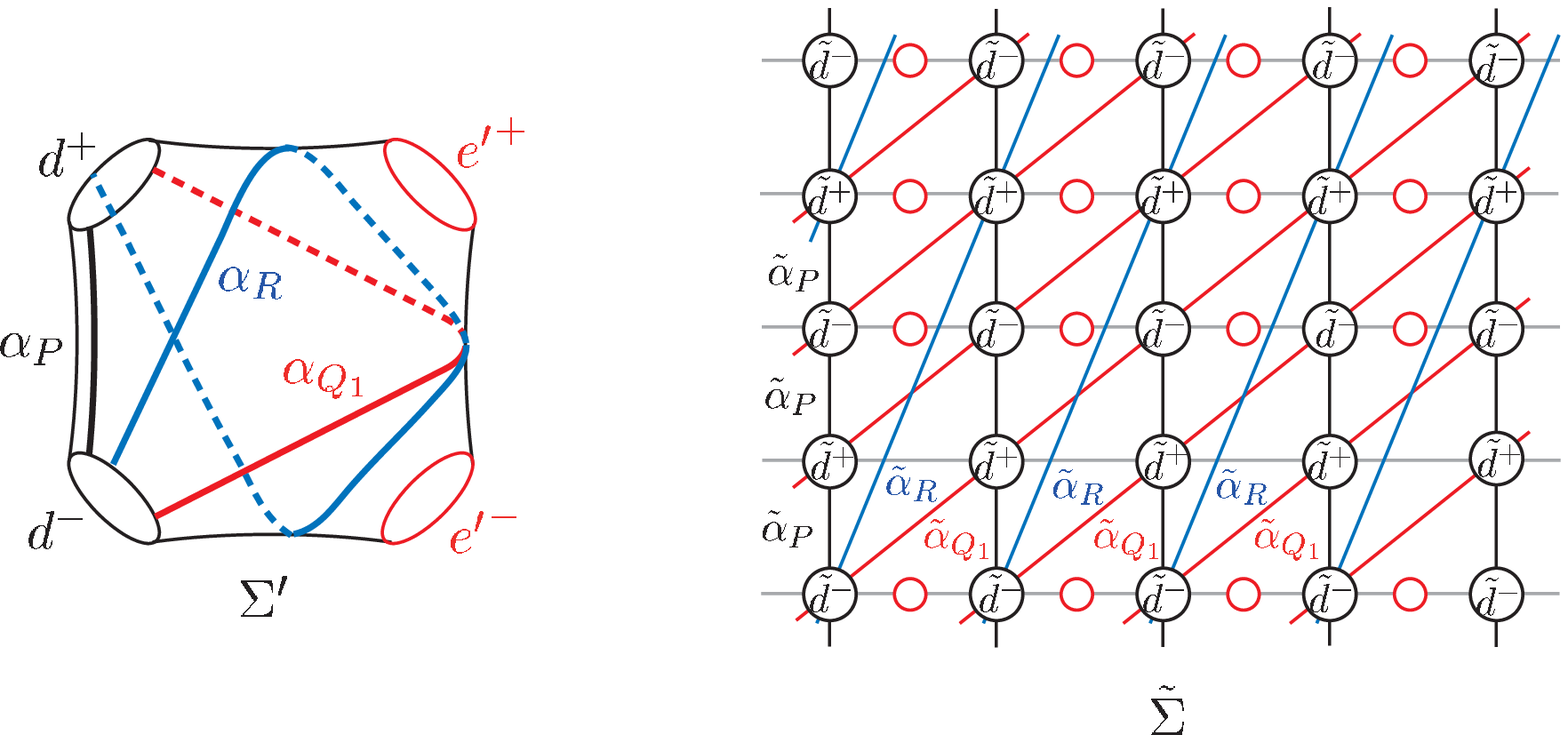}
\caption{}
\label{fig:universal_covering}
\end{center}
\end{figure}
See Figure \ref{fig:universal_covering}.
We note that, once we put a lift of $d^-$ at the origin $(0,0)$, 
the set of the coordinates corresponding to the lifts of $d^+$ 
is $\{ (s,t) \mid s \in \Integer, t \in \Integer_{\mathrm{odd}}  \}$, where 
$\Integer_{\mathrm{odd}}$ is the set of odd integers. 
For each arc connecting $d^+$ and $d^-$, we assign
the slope $s/t \in \Rational_{ \mathrm{odd}} \cup \{\infty\}$
of its preimage with respect to the above covering map, 
where $\Rational_{ \mathrm{odd}}$ is the set of irreducible rational numbers having odd numerators. 
Since the set of Haken spheres disjoint from $D \cup E'$
one-to-one corresponds to 
the set of simple arcs in $\Sigma'$ connecting $d^+$ and $d^-$
as in the proof of Lemma \ref{lem:sufficiency and necessity condition to be semi-primitive},
the above assignment provides an assignment of each vertex of $\mathcal{H}_D$ to an element of
$\Rational_{ \mathrm{odd}} \cup \{ 1/0 \}$.

We now briefly review some well-known facts on the Farey complex.
The {\it Farey complex} $\mathcal{F}$ is
the flag complex whose vertex set
is $\Rational \cup \{ 1/0 \}$.
Two vertices
$s_1 / t_1$ and $s_2 / t_2$ are
connected by an edge if and only if
$ s_1 t_2 - s_2 t_1  = \pm 1$.
See the left-hand side in Figure \ref{fig:farey_graph}.
\begin{figure}[htbp]
\begin{center}
\includegraphics[width=8.5cm,clip]{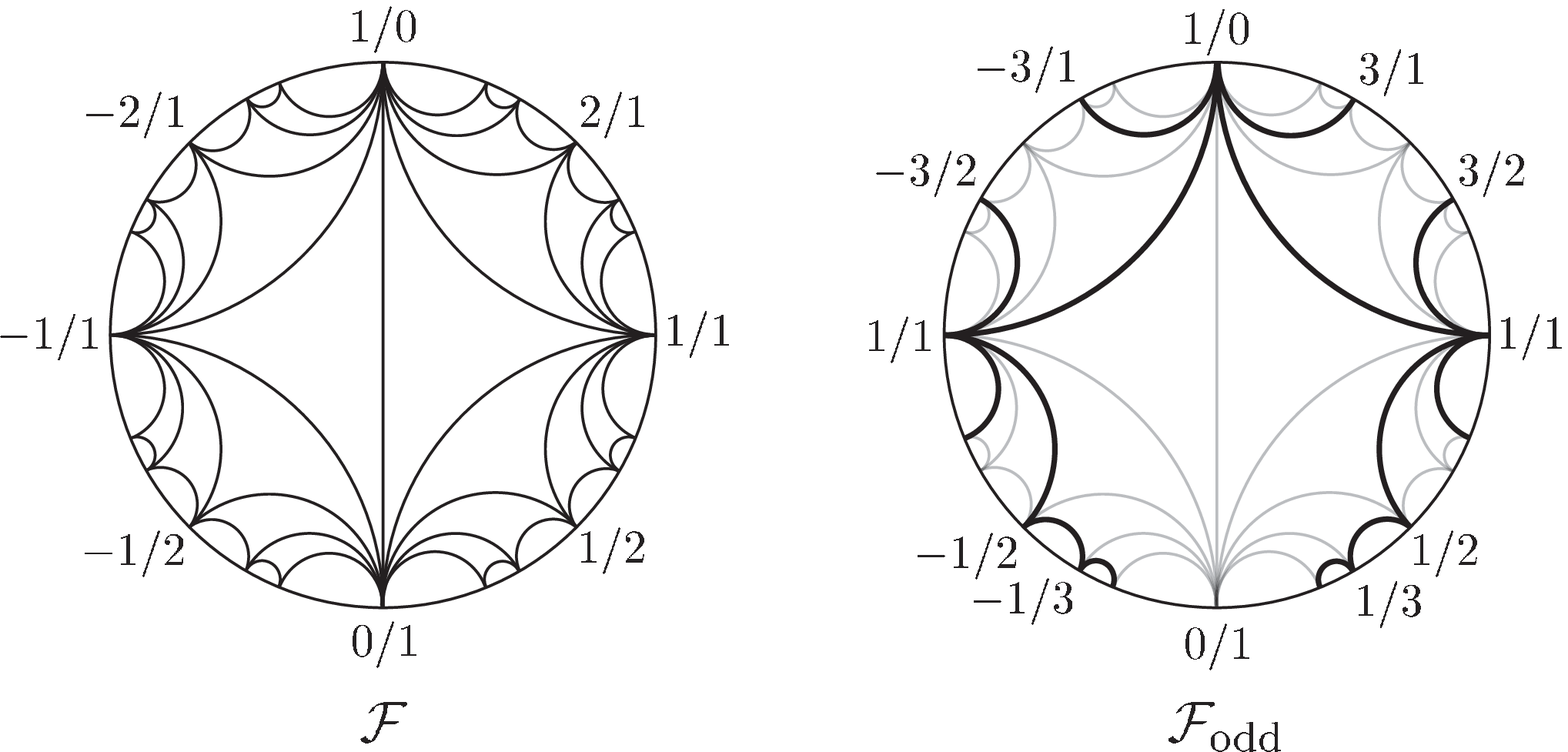}
\caption{}
\label{fig:farey_graph}
\end{center}
\end{figure}
The assignment of each vertex of $\mathcal{H}_{D}(V)$ with an element of
$\Rational \cup \{ 1/0 \}$ described above allows us to get an embedding
of $\mathcal{H}_{D}(V)$ into $\mathcal{F}$.
The image of $\mathcal{H}_{D}(V)$ is the full subcomplex
$\mathcal{F}_{\mathrm{odd}}$ of $\mathcal{F}$ spanned by
$\Rational_{ \mathrm{odd}} \cup \{ 1/0 \}$.
See the right-hand side in Figure \ref{fig:farey_graph}.
It is easy to check that $\mathcal{F}_{\mathrm{odd}}$ is a tree.
It follows that there exists a natural simplicial isomorphism from
$\mathcal{H}$ to the
simplicial complex obtained from
$\mathcal{SP}'(V)$ by replacing the star of each black vertex with
the tree simplicially isomorphic to $\mathcal{F}_{\mathrm{odd}}$.
See Figure \ref{fig:sphere_complex1}.
\begin{figure}[htbp]
\begin{center}
\includegraphics[width=11cm,clip]{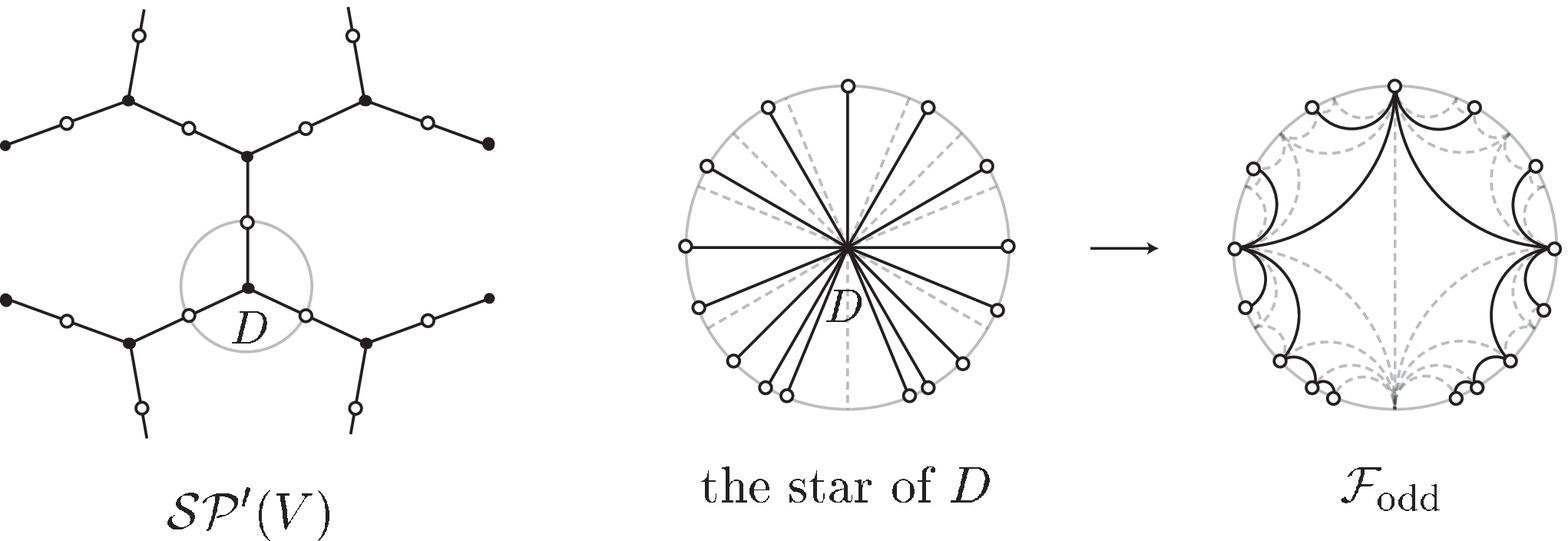}
\caption{}
\label{fig:sphere_complex1}
\end{center}
\end{figure}
Consequently, $\mathcal{H}$ is a tree.

Next, assume that $M_1 = S^2 \times S^1$.
Recall that, by Lemma
\ref{lem:the uniqueness of the reducing disk},
there exists the unique reducing disk $D$ in $V$.
Let $\Sigma_D$ be the $2$-holed torus obtained by cutting $\Sigma$
along $\partial D$.
Let $d^+$ and $d^-$ be the two components of $\partial \Sigma_D$.
Let $\mathcal{A} (\Sigma_D)$ be the simplicial complex
whose vertices are isotopy classes of simple arcs in $\Sigma_D$ connecting
$d^+$ and $d^-$ such that the collection of distinct $k+1$ vertices spans a $k$-simplex
if they admits a set of pairwise disjoint representatives.
Each simple arc $\alpha_P$ in $\Sigma_D$ connecting $d^+$ and $d^-$ determine
a unique Haken sphere $P$ of $(V, W; \Sigma)$.
By the uniqueness of $D$, this correspondence gives a simplicial isomorphism
$\mathcal{A}(\Sigma_D) \to \mathcal{H}$.
It is shown that $\mathcal{A}(\Sigma_D)$ is a contractible $3$-dimensional simplicial complex in
\cite{CMS09, Seo08}, and so is $\mathcal{H}$.
\end{proof}
We remark that the argument developed in \cite{McC91} allows us to show
easily that $\mathcal{H}$ is also a tree for the genus two Heegaard splitting
$(V, W; \Sigma)$ for $(S^2 \times S^1) \# (S^2 \times S^1)$.

\medskip

In the remaining of this section, we analyze the action of the Goeritz group
on the set of Haken spheres of
genus two Heegaard splittings for later works.


\begin{lemma}
\label{lem:two reducing spheres and transitivity}
Let $(V, W; \Sigma)$ be a genus two Heegaard splitting for
the connected sum of two lens spaces $L(p_1, q_1)$ and $L(p_2, q_2)$.
For any two Haken spheres $P$ and $Q$ of $(V, W; \Sigma)$ with
$|P \cap \Sigma \cap Q| = 4$,
there exists an element of the Goeritz group of $(V, W; \Sigma)$
that maps $P$ to $Q$.
\end{lemma}
\begin{proof}
The Haken sphere $P$ cuts $V$ into two solid tori $V_1$ and $V_2$, and
$W$ into $W_1$ and $W_2$.
We may assume that
$V_1 \cup W_1$ and $V_2 \cup W_2$ are punctured lens spaces.
Let $D$ and $E$ be the meridian disks of $V_1$ and $V_2$, respectively,
disjoint from $P$.
Similarly, let $D'$ and $E'$ be the meridian disks of $W_1$ and $W_2$, respectively,
disjoint from $P$.

\smallskip
\noindent {\it Claim.} Up to isotopy, $Q$ is disjoint from  $D \cup E'$ or $E \cup D'$.\\
{\it Proof of Claim.} Let $C_0$ be an outermost sub-disk of the disk $Q \cap V$ cut off by $P \cap Q \cap V$, which is contained in either $V_1$ or $V_2$.
Assume first that $C_0$ is contained in $V_1$.
Then there exists exactly one more such a sub-disk $C_1$ of $Q \cap V$, and it is also contained in $V_1$.
Since $|P \cap \Sigma \cap Q| = 4$, we have $V_1 \cap Q = C_1 \cup C_2$, and hence $Q$ is disjoint from $D$.
Further, if $D_0$ is an outermost sub-disk of the disk $Q \cap W$ cut off by $P \cap Q \cap W$, then $D_0$ must be contained in $W_2$, otherwise $\partial D$ would bound a meridian disk in $W_1$, which forms a non-separating sphere with the disk $D$ in the punctured lens space $V_1 \cup W_1$, a contradiction.
Further, by the same reason to the case of $C_0$ and $C_1$, there exists exactly one more sub-disk $D_1$ of $Q \cap W$, and it is also contained in $W_2$.
Thus $Q$ is also disjoint from $E'$.
If $C_0$ is contained in $V_2$, then, by the same argument, $Q$ is disjoint from $E \cup D'$.

\smallskip
By the claim, we assume that $Q$ is disjoint from $D \cup E'$ without loss of generality.
Let $\Sigma'$ be the 4-holed sphere
obtained by cutting $\Sigma$ along $\partial D \cup \partial E'$.
Let $d^+$ and $d^-$ (${e'}^+$ and ${e'}^-$, respectively)
be the two boundary circles of $\Sigma'$ coming from
$\partial D$ ($\partial E'$, respectively).
Then $P \cap \Sigma'$ ($Q \cap \Sigma'$, respectively)
is the frontier of a regular neighborhood
of the union of $d^+ \cup d^-$ and a simple arc $\alpha_P$ ($\alpha_Q$, respectively)
in $\Sigma'$ connecting $d^+$ and $d^-$.
Since $|P \cap \Sigma \cap Q| = 4$, we may assume that $\alpha_P \cap \alpha_Q = \emptyset$.
Set $\mu = P \cap \Sigma'$. 
Let $\nu$ be a simple closed curve in $\Sigma$ such that
$\nu$ separates $d^+ \cup {e'}^+$ and $d^- \cup {e'}^-$, and $\nu$ intersects $\mu$ transversely in two points.
See Figure \ref{fig:transitivity} (i).
\begin{figure}[htbp]
\includegraphics[width=13cm,clip]{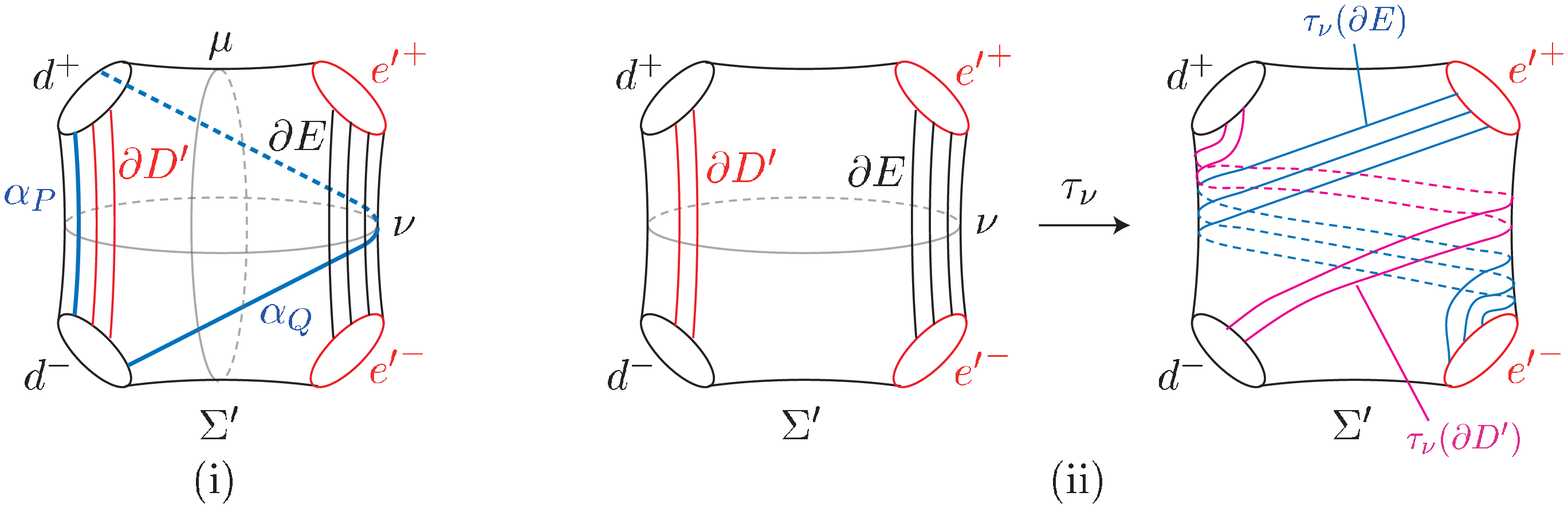}
\caption{}
\label{fig:transitivity}
\end{figure}
We note that a half-Dehn twist about $\mu$ extends to
an orientation-preserving homeomorphism of $L(p_1, q_1) \# L(p_2, q_2)$ that preserves $V$.
Up to a finite number of half-Dehn twists about $\mu$ and isotopy,
a single Dehn twist $\tau_{\nu}$ about $\nu$ maps $\alpha_P$ to $\alpha_Q$.
See Figure \ref{fig:transitivity} (ii).
However, $\tau_{\nu}$ extends to a homeomorphism of neither of
$V$ nor $W$. 
To see this, recall that each simple closed curve $l$ in $\Sigma$ determine 
a (possibly not reduced) word 
$w (l) $ on $\{ x, y \}$ ($\{ z, w \}$, respectively) that can be read off from the intersection of $l$ 
with $\partial D'$ and $\partial E'$ ($\partial D$ and $\partial E$, respectively) 
after fixing orientations of the simple closed curves. 
Note that this word gives the element of $\pi_1(W) = \langle x, y \rangle$ ($\pi_1 (V) = \langle z , w \rangle$, respectively)
represented by the loop $l$. 
On the surface $\Sigma'$, $\partial D'$ ($\partial E$, respectively) 
consists of $p_1$ ($p_2$, respectively) parallel simple arcs 
$\delta_1'$, $\delta_2' , \ldots, \delta_{p_1}'$ ($\epsilon_1$, $\epsilon_2 , \ldots, \epsilon_{p_2}$, respectively). 
Then the subword $w (\tau_{\nu} (\delta_{i}'))$ ($w (\tau_{\nu} (\epsilon_{j}))$, respectively) of 
$w (\tau_{\nu} (\partial D'))$ ($w ( \tau_{\nu} (\partial E))$, respectively) 
determined by the subarc $\tau_{\nu} (\delta_{i}')$ ($\tau_{\nu} (\epsilon_{j})$, respectively) of 
$\tau_{\nu} (\partial D')$ ($\tau_{\nu} (\partial E)$, respectively) 
is $x^{p_1}$ ($z^{p_2}$, respectively) for each $i \in \{ 1,2 , \ldots, p_1 \}$ ($j \in \{ 1,2 , \ldots, p_2 \}$, respectively). 
Here we move $\tau_{\nu} (\partial D')$ ($\tau_{\nu} (\partial E)$, respectively) slightly by isotopy so that 
$\tau_{\nu} (\partial D')$ ($\tau_{\nu} (\partial E)$, respectively) and $\partial D'$ 
($\partial E$, respectively) intersect each other transversely and minimally at points in the interior of $\Sigma'$. 
See the left-hand side in Figure \ref{fig:transitivity2}. 
This implies that $w( \tau_{\nu} (\partial D')) = x^{{p_1}^2}$ ($w( \tau_{\nu} (\partial E)) = x^{{p_2}^2}$, respectively). 
Thus $\tau_{\nu} (\partial D')$ ($\tau_{\nu} (\partial E)$, respectively) cannot bound a disk in $W$ (in $V$, respectively), 
and hence $\tau_{\nu}$ cannot extend to a homeomorphism of $V$ nor $W$.

But now we consider the composition $\tau_{\partial E'} \circ \tau_{\partial D} \circ \tau_{\nu}$. 
We may choose the Dehn twists $\tau_{\partial D}$ and $\tau_{\partial E'}$ so that 
the word $w( \tau_{\partial E'} \circ \tau_{\partial D} \circ \tau_{\nu} (\delta_i'))$ 
($w( \tau_{\partial E'} \circ \tau_{\partial D} \circ \tau_{\nu} (\epsilon_j))$, respectively)  
is an empty word after cancellation. See the right-hand side in Figure \ref{fig:transitivity2}. 
\begin{figure}[htbp]
\includegraphics[width=12cm,clip]{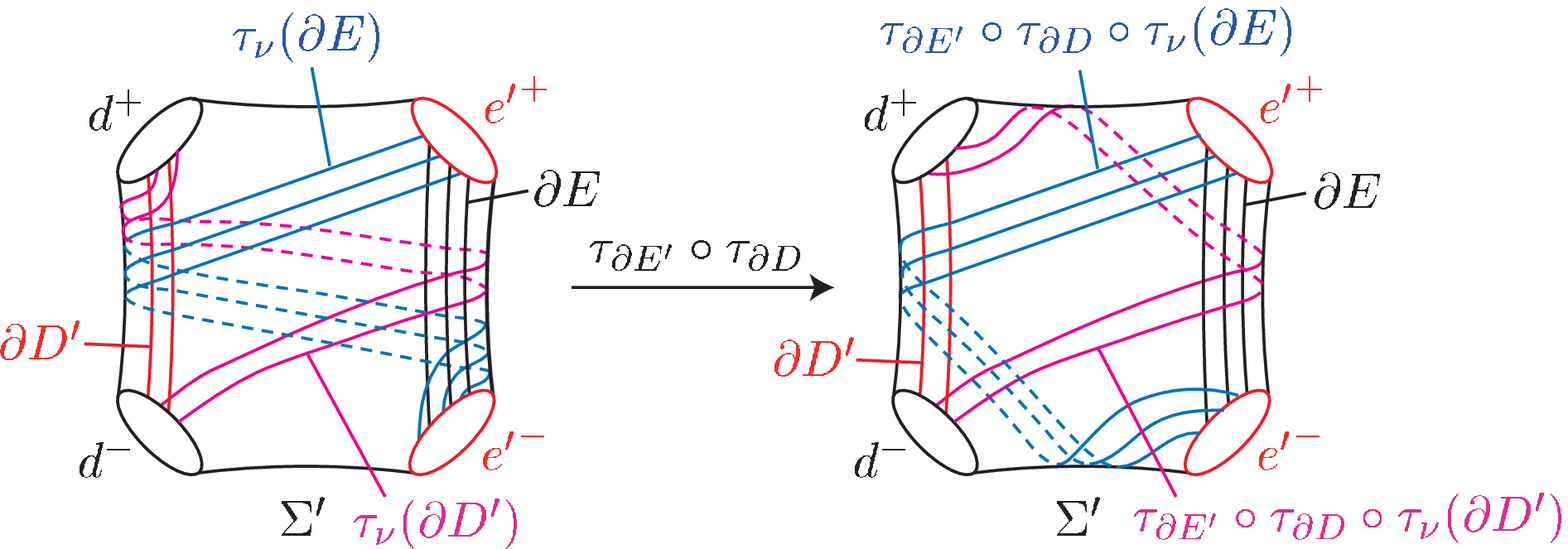}
\caption{}
\label{fig:transitivity2}
\end{figure}
This implies that the word 
$w( \tau_{\partial E'} \circ \tau_{\partial D} \circ \tau_{\nu} (\partial D'))$ 
($w( \tau_{\partial E'} \circ \tau_{\partial D} \circ \tau_{\nu} (\partial E))$, respectively) 
represents the trivial element of $\pi_1 (W)$ ($\pi_1 (V)$, respectively). 
Hence by Loop Theorem, $\tau_{\partial E'} \circ \tau_{\partial D} \circ \tau_{\nu} (\partial D')$
($\tau_{\partial E'} \circ \tau_{\partial D} \circ \tau_{\nu} (\partial E)$, respectively) 
bounds a disk in $W$ ($V$, respectively).  
Apparently, $\tau_{\partial E'} \circ \tau_{\partial D} \circ \tau_{\nu}$ fixes $\partial D$ 
and $\partial E'$. 
Consequently both
$\tau_{\partial E'} \circ \tau_{\partial D} \circ \tau_{\nu} (\partial D)$
and $\tau_{\partial E'} \circ \tau_{\partial D} \circ \tau_{\nu} (\partial E)$
bound disks in $V$.
Therefore by Alexander's trick, this composition extends to a homeomorphism of
$V$.
Similarly, $\tau_{\partial E'} \circ \tau_{\partial D} \circ \tau_{\nu} (\partial D')$ bounds a disk in $W$
and hence this composition  extends to a homeomorphism of $W$.
As a consequence, the map $\tau_{\partial E'} \circ \tau_{\partial D} \circ \tau_{\nu}$ extends to an orientation-preserving homeomorphism of $L(p_1, q_1) \# L(p_2, q_2)$ that preserves $V$.
\end{proof}

\begin{lemma}
\label{lem:transitivity on reducing spheres}
Let $M_1$ be a lens space or $S^2 \times S^1$, and let $M_2$ be a lens space.
Let $(V, W; \Sigma)$ be a genus two Heegaard splitting for $M_1 \# M_2$.
Then the Goeritz group of $(V, W; \Sigma)$ acts transitively
on the set of Haken spheres of $(V, W; \Sigma)$.
\end{lemma}
\begin{proof}
The case where $M_1$ is a lens space follows from
Theorem \ref{cor:contractibility of sphere complexes} and
Lemma \ref{lem:two reducing spheres and transitivity}.
Assume $M_1 = S^2 \times S^1$.
Let $P$ be a Haken sphere of $(V, W; \Sigma)$.
Then $P$ cuts $V$ into two solid tori $V_1$ and $V_2$, and
$W$ into $W_1$ and $W_2$.
We may assume that
$V_1 \cup W_1$ is a punctured $S^2 \times S^1$.
Let $D$ and $E$ be the meridian disks of $V_1$ and $V_2$, respectively,
disjoint from $P$.
Similarly, let $D'$ and $E'$ be the meridian disks of $W_1$ and $W_2$, respectively,
disjoint from $P$. 	
In this case, we may assume that $\partial D = \partial D'$.
As we have seen in Subsection
\ref{subsec:Semi-primitive disks and a genus two Heegaard splitting for
the connected sum of a lens space and S2 times S1},
$D$ is the unique reducing disk in $V$.
Let $\Sigma_D$ be a 2-holed torus
obtained by cutting $\Sigma$ along $\partial D$.
We denote the boundary circles of $\Sigma_D$ by $d^+$ and $d^-$.
Then there exists a simple arc $\alpha_P$ in $\Sigma_D$
connecting $d~+$ and $d^-$ such that
$P \cap \Sigma_D$ is the frontier of a regular neighborhood of
$d^+ \cup \alpha_P \cup d^-$.
Let $Q$ be another Haken sphere of $(V, W; \Sigma)$.
By the same argument as above,
there exists a simple arc $\alpha_Q$ in $\Sigma_D$
connecting $d~+$ and $d^-$ such that
$Q \cap \Sigma_D$ is the frontier of a regular neighborhood of
$d^+ \cup \alpha_Q \cup d^-$.
Then there exists a hoeomorphism $\varphi$ of $\Sigma_D$
defined by pushing $d^+$ in such a way that $\varphi$ maps $\alpha_P$ to $\alpha_Q$.
See Figure \ref{fig:twice-punctured_torus}.
\begin{figure}[htbp]
\includegraphics[width=7cm,clip]{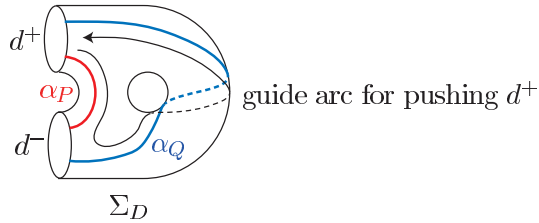}
\caption{Pushing $d^+$ along the guide arc maps $\alpha_P$ to $\alpha_Q$.}
\label{fig:twice-punctured_torus}
\end{figure}
The homeomorphism $\varphi$ extends to a slide of a foot of a handle of each of $V$ and $W$, and so
$\varphi$ extends to a homeomorphism of $M_1 \# M_2$ 
that preserves $V$ and $W$, and takes $P$ to $Q$ up to isotopy. 
\end{proof}

\section{Classification of genus two Heegaard splittings}
\label{sec:Classification of genus two Heegaard splittings}

For  $i \in \{1, 2\}$, let $M_i$ be a lens space $L(p_i,q_i)$ or $S^2 \times S^1$, and let $(V_i, W_i; \Sigma_i)$ be a
genus one Heegaard splitting for $M_i$.
By \cite{BO83},
$\Sigma_i$ is the unique genus one Heegaard surface for $M_i$ up to isotopy, and
there exists an orientation-preserving homeomorphism of $M_i$ that interchanges
$V_i$ and $W_i$ if and only if ${q_i}^2 \equiv 1 \pmod {p_i}$ or $M_i = S^2 \times S^1$.
Let $B_i$ be a 3-ball embedded in $M_i$
so that $B_i \cap \Sigma_i$ is a single disk properly embedded in $B_i$.
A genus two Heegaard splitting $(V, W, \Sigma)$ for $M_1 \# M_2$ is created
by gluing $V_1$ and $V_2$ to obtain $V$, and $W_1$ and $W_2$ to obtain $W$,
by an appropriate orientation-reversing map $\partial B_1 \to \partial B_2$
after removing the interiors of $B_1$ and $B_2$ from $M_1$ and $M_2$, respectively.
Also, another genus two Heegaard splitting $(V', W'; \Sigma')$ for $M_1 \# M_2$ is created
by gluing $V_1$ and $W_2$ to obtain $V'$, and $W_1$ and $V_2$ to obtain $W'$
in the same way.
From \cite{Hak68}, it is known that each genus two Heegaard surface for $M_1 \# M_2$ is
one of the above two Heegaard surfaces $\Sigma$ and $\Sigma'$
modulo the homeomorphisms of $M_1 \# M_2$.
However, it is shown in \cite{Bir75} that $\Sigma$ and $\Sigma'$ do not always
coincide modulo homeomorphisms of $M$.
In \cite{MS88}, genus two Heegaard surfaces for $L(p_1,q_1) \# L(p_2, q_2)$ modulo the
homeomorphisms of $L(p_1,q_1) \# L(p_2, q_2)$ are classified when $p_1 = p_2$ as follows.
\begin{theorem}[\cite{MS88}]
\label{thm:Montesinos and Safont}
Let $M$ be the connected sum of two lens spaces $L(p,q_1)$ and $L(p, q_2)$.
Then there exists a unique genus two Heegaard surface for $M$ modulo homeomorphisms of $M$
if and only if ${q_1}^2 \equiv 1$ or ${q_2}^2 \equiv 1 \pmod p$, and two Heegard surfaces otherwise.
\end{theorem}

The following is a generalization of Theorem \ref{thm:Montesinos and Safont} to the case of all non-prime $3$-manifolds which admit genus two Heegaard splittings.

\begin{theorem}
\label{thm:Generalization of Montesinos-Safont theorem}
Let $M$ be a connected sum of $M_1$ and $M_2$, where $M_i$ is a lens space $L(p_i,q_i)$ or $S^2 \times S^1$ for $i \in \{1, 2\}$. 
Then there exists a unique genus two Heegaard surface for $M$ modulo homeomorphisms of $M$
if and only if one of $M_i$ is $S^2 \times S^1$ or a lens space $L(p_i, q_i)$ with ${q_i}^2 \equiv 1 \pmod {p_i}$,
and two Heegard surfaces otherwise.
\end{theorem}
\begin{proof}
The ``if" part follows trivially from the descriptive comments at the beginning of this section. 

Now we prove  the ``only if" part. 
Suppose that both of $M_i$ are lens spaces $L(p_i, q_i)$ with ${q_i}^2 \not\equiv 1 \pmod {p_i}$.
Let $(V, W; \Sigma)$ and $(V', W'; \Sigma')$ be the two Heegaard splittings
of $M = L(p_1, q_1) \# L(p_2, q_2)$ obtained from the genus one Heegaard splittings
$(V_i, W_i; \Sigma_i)$ of $L(p_i, q_i)$, $i=1,2$,
as described in the beginning of this section.
The construction provides the Haken spheres $P$ and $P'$ for
the splittings $(V, W; \Sigma)$ and $(V', W'; \Sigma')$, respectively.
We give an orientation of $P$ and $P'$ so that
the $L(p_1, q_1)$-summand lies in the negative side.
Suppose that there exists a homeomorphism $f$ of $M$ that maps $\Sigma'$ to $\Sigma$.
Then, by Lemma \ref{lem:transitivity on reducing spheres}, there exists
a homeomorphism $g$ of $M$ that preserves $\Sigma$ and that maps
$f(P')$ to $P$.
We may assume that $P$ and $g \circ f(P')$ have the same orientation.
Moreover, we may assume that $g \circ f$ induces a homeomorphism of $L(p_1, q_1)$ that preserves
$V_1$ and $W_1$.
Since ${q_1}^2 \not\equiv 1 \pmod {p_1}$, $g \circ f$ is orientation-preserving.
Now $g \circ f$ induces an orientation-preserving
homeomorphism of $L(p_2, q_2)$ that interchanges
$V_2$ and $W_2$, contradicting ${q_2}^2 \equiv 1 \pmod {p_2}$.
\end{proof}

\section{Genus two Goeritz groups}
\label{sec:Genus two Goeritz groups}

Let $(V, W; \Sigma)$ be a genus two Heegaard splitting for the
connected sum of two lens spaces.
A Haken sphere $P$ of $(V, W; \Sigma)$ is said to be {\it reversible}
if there exists an element $g$ of $\mathcal{G}$ fixing $P$ setwise
such that $g$ restricted to $P$ is an orientation-reversing
homeomorphism on $P$.
We say that the splitting $(V, W; \Sigma)$
is {\it symmetric} if it admits a reversible
Haken sphere.
By Lemma \ref{lem:transitivity on reducing spheres},
if the splitting $(V, W; \Sigma)$ admits a reversible Haken sphere,
then every Haken sphere of $(V, W; \Sigma)$ is reversible.

For a genus two Heegaard splitting $(V, W; \Sigma)$ for
the connected sum of two lens spaces, we fix the following notations throughout the section.
\begin{itemize}
\item
Disjoint, non-parallel semi-primitive disks $D$ and $E$ in $V$,
\item
the disjoint semi-primitive disks $D'$ and $E'$ in $W$ such that $D \cap E' = E \cap D' = \emptyset$
(such $D'$ and $E'$ are determined uniquely by Lemma \ref{lem:uniqueness of a disjoint semi-primitive disk}),
\item
the Haken sphere $P$ of $(V, W; \Sigma)$ disjoint from $D \cup E$
(the existence and uniqueness of
$P$ follows from Lemma \ref{lem:disjoint semi-primitive disks and reducing spheres}), and
\item
a Haken sphere $Q_1$ ($Q_2$, respectively)  of $(V, W; \Sigma)$ disjoint from
$D \cup E'$ ($E \cup D'$, respectively) such that
$|P \cap \Sigma \cap Q_1| = 4$ ($|P \cap \Sigma \cap Q_2| = 4$, respectively)
(the existence of $Q_1$ and $Q_2$ follows from the
proof of Lemma \ref{lem:sufficiency and necessity condition to be semi-primitive}).
\end{itemize}
\begin{figure}[htbp]
\begin{center}
\includegraphics[width=14cm,clip]{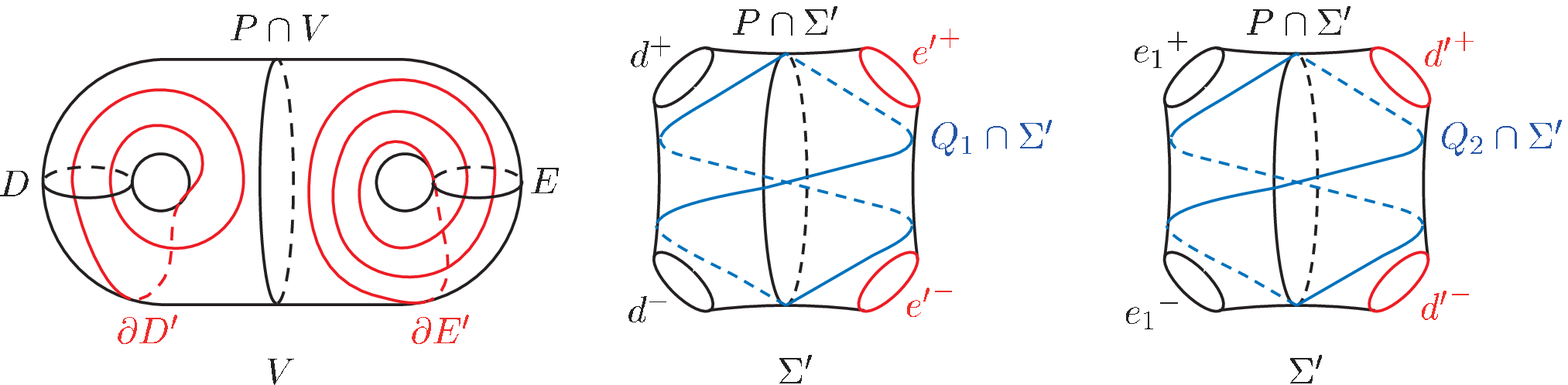}
\caption{}
\label{fig:fixed_disks_for_two_lens_spaces}
\end{center}
\end{figure}
See Figure \ref{fig:fixed_disks_for_two_lens_spaces}.
In the figure, the 4-holed sphere $\Sigma'$ is
obtained by cutting $\Sigma$ along $\partial D \cup \partial E'$ and
the boundary circles $d^+$ and $d^-$ (${e'}^+$ and ${e'}^-$, respectively)
come from $\partial D$ ($\partial E'$, respectively).
By $\alpha \in \mathcal{G}$, we denote the hyperelliptic involution
of both $V$ and $W$.
By $\beta \in \mathcal{G}$, we denote the extension of a half-Dehn twist about the disk $P \cap V$.
By $\gamma_1 \in \mathcal{G}$ ($\gamma_2 \in \mathcal{G}$, respectively), we denote an element of order 2
that preserves $D \cup E'$ ($E \cup D'$, respectively) and that interchanges $P$ and $Q_1$ ($P$ and $Q_2$, respectively)
(the existence of this element
will be proved in Lemma \ref{lem:stabilizer of a pair of reducing spheres}).
When $P$ is reversible, we denote by $\delta \in \mathcal{G}$ an element of order 2
that reverses $P$.

Also, for a genus two Heegaard splitting $(V, W; \Sigma)$ for
the connected sum of $S^2 \times S^1$ and a lens space,
we fix the following notations throughout the section.
\begin{itemize}
\item
The reducing disk $D$ in $V$ and the disk $D'$ in $W$ bounded by $\partial D$
($D$ is unique by Lemma \ref{lem:the uniqueness of the reducing disk}),
\item
disjoint, non-isotopic, semi-primitive disks $E_1$ and $E_2$ in $V$,
\item
a semi-primtive disk $E'$ in $W$ such that $\partial E'$ has the same type
with respect to $E_1$ and $E_2$ (the existence of $E'$ follows from the proof
of Lemma \ref{lem:semi-primitive disks are interchangible}), and
\item
Haken spheres $P$ and $Q$ of $(V, W; \Sigma)$ disjoint from $D \cup E_1$ such that
$|P \cap \Sigma \cap Q| = 4$
(the existence of $P$ and $Q$ follows
from the proof of
Lemma \ref{lem:the uniqueness of the reducing disk} (1)).
\end{itemize}
\begin{figure}[htbp]
\begin{center}
\includegraphics[width=9cm,clip]{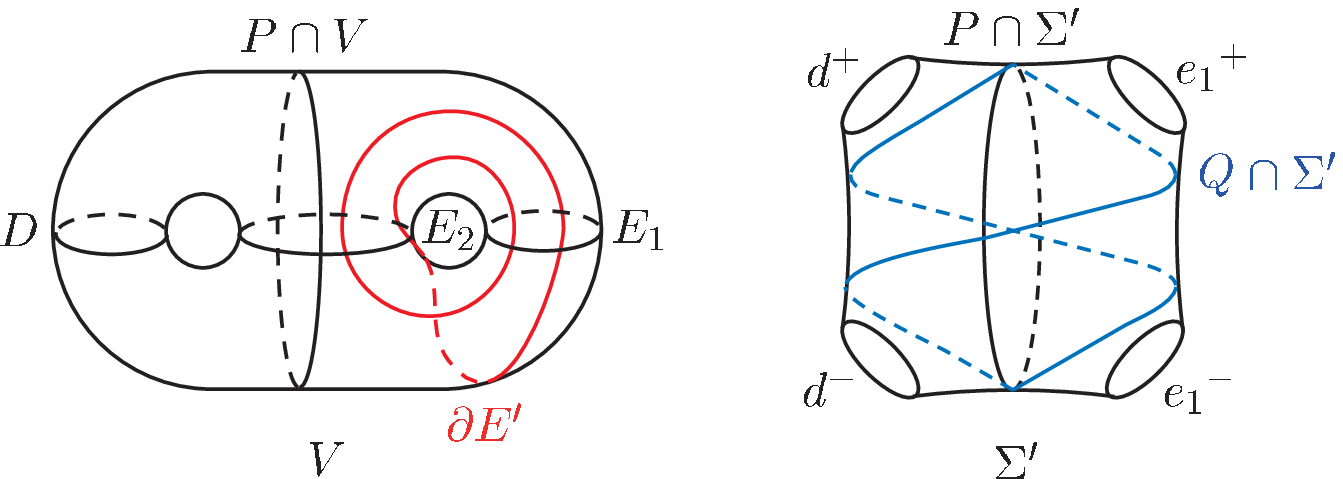}
\caption{}
\label{fig:fixed_disks_for_a_lens_space_and_sphere_bundle}
\end{center}
\end{figure}
See Figure \ref{fig:fixed_disks_for_a_lens_space_and_sphere_bundle}.
In the figure, the 4-holed sphere $\Sigma'$ is
obtained by cutting $\Sigma$ along $\partial D \cup \partial E_1$ and
the boundary circles $d^+$ and $d^-$ (${e_1}^+$ and ${e_1}^-$, respectively)
come from $\partial D$ ($\partial E_1$, respectively).
By $\alpha \in \mathcal{G}$, we denote the hyperelliptic involution
of both $V$ and $W$.
By $\beta \in \mathcal{G}$ ($\tau \in \mathcal{G}$, respectively),
we denote the extension of a half-Dehn twist (Dehn twist, respectively)
about the disk $P \cap V$ ($D$, respectively).
By $\gamma \in \mathcal{G}$, we denote an element of order 2
that interchanges $P$ and $Q$ (the existence of this element
is proved in Lemma \ref{lem:stabilizer of a pair of reducing spheres}).
By $\sigma \in \mathcal{G}$, we denote an element of order 2
that preserve $E'$ and that interchanges $E_1$ and $E_2$
(the existence of this element
will be proved in Lemma \ref{lem:semi-primitive disks are interchangible}).

Now we are ready to state the main theorem, which provides presentations of
genus two Goeritz groups of all non-prime 3-manifolds.
(Recall that the genus two Goeritz group of $(S^2 \times S^1) \# (S^2 \times S^1)$ is
the mapping class group of the genus two handlebody and its presentation
is already known.)

\begin{theorem}
\label{thm:presentations of the Goeritz groups}
Let $M_1$ be a lens space or $S^2 \times S^1$, and let $M_2$ be a lens space.
Let $(V, W; \Sigma)$ be a genus two Heegaard splitting for $M_1 \# M_2$.
Then the Goeritz group $\mathcal{G}$ of $(V, W; \Sigma)$
has the following presentation:
\begin{enumerate}
\item
If $M_1$ is a lens space,
\begin{enumerate}
\item
\label{thm:lens and lens not reversible}
$\langle \alpha \mid \alpha^2 \rangle \oplus
\langle \beta, \gamma_1, \gamma_2
\mid {\gamma_1}^2, {\gamma_2}^2 \rangle$ if $(V, W; \Sigma)$ is not symmetric;
\item
\label{thm:lens and lens reversible}
$\langle \alpha \mid \alpha^2 \rangle \oplus
\langle \beta, \gamma_1, \delta
\mid {\gamma_1}^2, \delta^2, \delta \beta \delta = \alpha \beta \rangle$ if
$(V, W; \Sigma)$ is symmetric;
\end{enumerate}
\item
\label{thm:lens and S2 S1}
If $M_1 = S^2 \times S^1$,
$\langle \alpha \mid \alpha^2 \rangle \oplus \langle \beta, \gamma, \sigma \mid \gamma^2, \sigma^2 \rangle
\oplus \langle \tau  \rangle$.
\end{enumerate}
\end{theorem}
We remark that, from Section \ref{sec:Classification of genus two Heegaard splittings},
once a genus two Heegaard splitting for $M = L(p_1, q_1) \# L(p_2, q_2)$ is given,
we may easily determine whether the splitting is symmetric or not.
If $L(p_1, q_1) \not\cong L(p_2, q_2)$ (as oriented manifolds), no genus two Heegaard splitting
of $M$ is symmetric.
If $L(p_1, q_1) \cong L(p_2, q_2)$, exactly one genus two Heegaard splitting
of $M$ is symmetric and the other, if any, is not.

Throughout the section,
for suitable subsets $A_1$, $A_2, \ldots , $ $A_k$ of $M_1 \# M_2$,
we denote by $\mathcal{G}_{\{ A_1, A_2, \ldots , A_k \}}$ the subgroup of
its Goeritz group $\mathcal{G}$ consisting of elements that preserve each of
$A_1$, $A_2, \ldots , $ $A_k$ setwise.


\begin{lemma}
\label{lem:stabilizer of a reducing sphere}
Let $M_1$ be a lens space or $S^2 \times S^1$, and let $M_2$ be a lens space.
Let $(V, W; \Sigma)$ be a genus two Heegaard splitting for $M_1 \# M_2$.
\begin{enumerate}
\item
If $M_1$ is a lens space, then
$\mathcal{G}_{\{D, P\}} =
\langle \alpha \mid \alpha^2 \rangle \oplus \langle \beta \rangle $.
\item
If $M_1$ is $S^2 \times S^1$, then
$\mathcal{G}_{\{D, P\}} =
\langle \alpha \mid \alpha^2 \rangle \oplus \langle \beta \rangle
\oplus \langle \tau \rangle$.
\end{enumerate}
\end{lemma}
\begin{proof}
Let $g$ be an element of $\mathcal{G}_{\{D, P\}}$.

\noindent (1)
Since $g$ preserves $D$, $g$ is orientation-preserving on
$P$.
We may assume that $g$ maps each of the disks
$D$, $D'$, $E$ and $E'$ to itself.
Moreover if $g$ is orientation-preserving on
$D$ ($E$, respectively), then so is on $D'$ ($E'$, respectively).
Hence by taking a composition with $\alpha$ and $\beta$,
if necessary, we may assume that
$g$ fixes $D \cup D' \cup E \cup E'$.
Now, $\Sigma$ cut off by $D \cup D' \cup E \cup E'$ consists of
several disks and a single annulus.
By Alexander's trick, boundary-preserving homeomorphisms
on a disk is unique up to isotopy.
Also, boundary-preserving homeomorphisms
on an annulus are determined by Dehn twist
about its core circle up to isotopy.
This implies that $g$ is a power of $\beta$.

\noindent (2)
Let $l$ be a simple closed curve in $\Sigma$ disjoint from $P$
that intersects $\partial D$ in a single point.
Let $g$ be an element of $\mathcal{G}_{\{D, P\}}$.
Since $g$ preserves $D$, $g$ is orientation-preserving on
$P$.
We may assume that $g$ maps each of the disks
$D$, $D'$, $E$ and $E'$ to itself.
Moreover if $g$ is orientation-preserving on
$D$ ($E$, respectively), then so is on $D'$ and $l$ ($E'$ and $l$, respectively).
Hence modulo the action of $\alpha$ and $\tau$,
$g$ fixes $D \cup D' \cup l \cup E \cup E'$.
The remaining argument is exactly the same as (1).
\end{proof}

\begin{lemma}
\label{lem:transitivity on vertices connected to a vertex}
Let $M_1$ be a lens space or $S^2 \times S^1$, and let $M_2$ be a lens space.
Let $(V, W; \Sigma)$ be a genus two Heegaard splitting for $M_1 \# M_2$.
\begin{enumerate}
\item
Suppose that $M_1$ is a lens space.
Let
$Q'_1$ be a Haken sphere of $(V,W;\Sigma)$ disjoint from
$D \cup E'$ such that
$|P  \cap \Sigma \cap Q'_1| = 4$.
Then a  power of $\beta$ maps $Q_1$ to $Q'_1$.
\item
Suppose that $M_1$ is $S^2 \times S^1$.
Let $Q'$ be a Haken sphere of $(V,W;\Sigma)$ disjoint from
$D \cup E$ such that
$|P  \cap \Sigma \cap Q'| = 4$.
Then a power of $\beta$ maps $Q$ to $Q'$.
\end{enumerate}
\end{lemma}
\begin{proof}
(1) Let $\Sigma'$ be the 4-holed sphere
obtained by cutting $\Sigma$ along $\partial D \cup \partial E'$.
Let $d^+$ and $d^-$ (${e'}^+$ and ${e'}^-$, respectively)
be the two boundary circles of $\Sigma'$ coming from
$\partial D$ ($\partial E'$, respectively).
Let $\alpha_P$, $\alpha_{Q_1}$ and $\alpha_{Q_1'}$ be simple arcs in $\Sigma'$ connecting $d^+$ and $d^-$
such that the frontiers of regular neighborhoods of
$d^+ \cup \alpha_P \cup d^-$, $d^+ \cup \alpha_{Q_1} \cup d^-$ and $d^+ \cup \alpha_{Q_1'} \cup d^-$
are $P \cap \Sigma$, $Q_1 \cap \Sigma$ and $Q'_1 \cap \Sigma$,
respectively.
We may assume that $\alpha_P \cap \alpha_{Q_1} = \alpha_P \cap \alpha_{Q_1'} = \emptyset$ since
$|P  \cap \Sigma \cap Q_1| =  |P  \cap \Sigma \cap Q'_1| = 4$.
Since $\alpha_P$ cuts $\Sigma'$ into a pair of pants,
a certain power of $\beta$ carries $\alpha_{Q_1}$ to $\alpha_{Q_1'}$.
The proof of (2) is exactly the same as (1).
\end{proof}

\begin{lemma}
\label{lem:stabilizer of a pair of reducing spheres}
Let $M_1$ be a lens space or $S^2 \times S^1$, and let $M_2$ be a lens space.
Let $(V, W; \Sigma)$ be a genus two Heegaard splitting for $M_1 \# M_2$.
\begin{enumerate}
\item
If $M_1$ is a lens space, then
$\mathcal{G}_{\{D, P, Q\}} = \langle \alpha \mid \alpha^2 \rangle$,
and $\mathcal{G}_{\{D, P \cup Q\}} = \langle \alpha \mid \alpha^2 \rangle \oplus
\langle \gamma_1 \mid {\gamma_1}^2 \rangle$.
\item
If $M_1$ is $S^2 \times S^1$, then
$\mathcal{G}_{\{D, P, Q\}} = \langle \alpha \mid \alpha^2 \rangle \oplus \langle \tau \rangle$,
and $\mathcal{G}_{\{D, P \cup Q\}} = \langle \alpha \mid \alpha^2 \rangle \oplus
\langle \gamma \mid \gamma^2 \rangle  \oplus \langle \tau \rangle$.
\end{enumerate}
\end{lemma}
\begin{proof}
(1) We first show the existence of the element $\gamma_1 \in \mathcal{G}$.
Let $\beta'_1$ denote a half-Dehn twist about the sphere $Q_1$.
By Lemma \ref{lem:transitivity on reducing spheres},
there exists an element $g \in \mathcal{G}$ that carries $P$ to $Q_1$.
We may assume without loss of generality that $g$ maps $D$ to $D$ and $E'$ to $E'$.
By Lemma \ref{lem:transitivity on vertices connected to a vertex},
a certain power ${\beta'_1}^n$ of $\beta'_1$ carries $g(Q_1)$ to $P$.
We remark that ${\beta'_1}^n \circ g$ interchanges $P$ and $Q_1$ and
this map carries $D$ to $D$ and $E'$ to $E'$.
Up to isotopy, we may assume that $({\beta'_1}^n \circ g)^2$ fixes $D \cup E' \cup P \cup Q_1$.
Then by cutting $\Sigma$ along $\partial D \cup \partial E'$ and considering
simple arcs connecting the two holes coming from $\partial D$ as in the proof of
Lemma \ref{lem:transitivity on vertices connected to a vertex},
we can easily check that $({\beta'_1}^n \circ g)^2$ restricted to $\Sigma$
is a power of Dehn twist along $\partial E'$.
Hence $({\beta'_1}^n \circ g)^2$ is isotopic to the identity.
This implies that ${\beta'_1}^n \circ g$ is the required element $\gamma_1$.
Since $\tau$ is commutative with any element of $\mathcal{G}$ that preserves
$D$, (2) follows from the same argument as (1).
\end{proof}

\begin{lemma}
\label{lem:stabilizer of a semi-primitive disk}
Let $M_1$ be a lens space or $S^2 \times S^1$, and let $M_2$ be a lens space.
Let $(V, W; \Sigma)$ be a genus two Heegaard splitting for $M_1 \# M_2$.
Let $D$ be a semi-primitive disk in $V$.
\begin{enumerate}
\item
If $M_1$ is a lens space,
then $\mathcal{G}_{\{D\}} =
\langle \alpha \mid \alpha^2 \rangle \oplus \langle \beta, \gamma \mid \gamma^2 \rangle$.
\item
If $M_1$ is $S^2 \times S^1$,
then $\mathcal{G}_{\{ D , E_1 \}} =
\langle \alpha \mid \alpha^2 \rangle \oplus \langle \beta, \gamma \mid \gamma^2 \rangle
\oplus \langle \tau \rangle$.
\end{enumerate}
\end{lemma}
\begin{proof}
(1) By Lemma \ref{lem:uniqueness of a disjoint semi-primitive disk},
$E'$ is the unique non-separating disk in $W$ disjoint from
$D$.
This implies that each element of $\mathcal{G}_D$ preserves $E'$.
Let $\Sigma'$ be the 4-holed sphere
obtained by cutting $\Sigma$ along $\partial D \cup \partial E'$.
Let $d^+$ and $d^-$ (${e'}^+$ and ${e'}^-$, respectively)
be the two boundary circles of $\Sigma'$ coming from
$\partial D$ ($\partial E'$, respectively).
As in the proof of Theorem \ref{cor:contractibility of sphere complexes},
let $\mathcal{H}_{D}$ be the full subcomplex of the complex $\mathcal{H}$
of Haken spheres of $(V, W; \Sigma)$ spanned by the vertices corresponding to Haken spheres disjoint from $D$.
Then $\mathcal{H}_{D}$ is a tree as we have seen in
Lemma \ref{cor:contractibility of sphere complexes}.
Let $\mathcal{H}'_{D}(V)$ be the first barycentric subdivision of
$\mathcal{G}_{D}$.
The group $\mathcal{G}_{D}$ acts on $\mathcal{H}'_{D}(V)$ simplicially.
Moreover, the quotient of $\mathcal{H}'_{D}(V)$ by the action of
$\mathcal{G}_{D}$ is a single edge.
Then by the Bass-Serre theory on groups acting on trees \cite{Ser77}, we have
$\mathcal{G}_{\{D\}} =
\mathcal{G}_{\{D, P \}} *_{\mathcal{G}_{\{D, P, Q_1 \}}}
\mathcal{G}_{\{D, P \cup Q_1 \}} $.
Now, (1) follows from Lemmas \ref{lem:stabilizer of a reducing sphere} and
\ref{lem:stabilizer of a pair of reducing spheres}.

\noindent (2)
Cutting $\Sigma$ along $D \cup E_1$ instead of $D \cup E'$,
we get the presentation by almost the same argument as (1).
\end{proof}

\begin{lemma}
\label{lem:semi-primitive disks are interchangible}
Let $(V, W; \Sigma)$ be the genus two Heegaard splitting for the
connected sum of $S^2 \times S^1$ and a lens space.
Let $E_1$ and $E_2$ be disjoint, non-isotopic, semi-primitive and non-reducing disks in $V$.
Then there exists an element of the Goeritz group
$\mathcal{G}$ of the Heegaard splitting $(V, W; \Sigma)$ that interchanges $E_1$ and $E_2$.
\end{lemma}
\begin{proof}
It is easy to see that
there exists a non-reducing semi-primitive disk $\widehat{E}_2$ in $V$ such that
$E_1$ and $\widehat{E}_2$ can be interchanged by an element of $\mathcal{G}$.
Thus it suffices to show that there exists an element of $\mathcal{G}$
that preserves $E_1$ and that maps $E_2$ to $\widehat{E}_2$.
Let $\Sigma_D$ be a 2-holed torus obtained by cutting $\Sigma$
along $\partial D$. We denote the boundary circles of $\Sigma_D$
by $d^+$ and $d^-$. 	
Since both $E_2$ and $\widehat{E}_2$ are meridian disks of
the solid torus obtained by cutting
$V$ along $D$,
there exists a pushing of $d^+$ in $\Sigma_D$ that preserve $\partial E_1$,
and that maps $\partial E_2$ to $\partial \widehat{E}_2$.
As we have seen in Lemma \ref{lem:transitivity on reducing spheres},
every pushing map of $d^+$ extends to a slide of a foot of a handle of each of $V$ and $W$,
thus it extends to a homeomorphism of $(S^2 \times S^1) \# L(p, q)$ that preserves $V$.
\end{proof}

Finally, the following two lemmas follow from
Lemmas \ref{lem:disjoint semi-primitive disks and reducing spheres} and
\ref{lem:stabilizer of a semi-primitive disk}.

\begin{lemma}
\label{lem:stabilizers of semi-primitive pairs}
Let $(V, W; \Sigma)$ be a genus two Heegaard splitting for the
connected sum of two lens spaces.
\begin{enumerate}
\item
If $(V, W; \Sigma)$ is not symmetric, then
$\mathcal{G}_{\{ D, E \}} = \mathcal{G}_{\{ D \cup E \}} =
\langle \alpha \mid \alpha^2 \rangle \oplus \langle \beta \rangle$ .
\item
If $(V, W; \Sigma)$ is symmetric, then
$\mathcal{G}_{\{ D, E \}} =
\langle \alpha \mid \alpha^2 \rangle \oplus \langle \beta \rangle$
and
$\mathcal{G}_{\{ D \cup E \}} =
\langle \alpha \mid \alpha^2 \rangle \oplus \langle \beta , \delta \mid \delta^2,
\delta \beta \delta = \alpha \beta \rangle$.
\end{enumerate}
\end{lemma}

\begin{lemma}
\label{lem:stabilizers of non-separating pairs}
Let $(V, W; \Sigma)$ be the genus two Heegaard splitting for the
connected sum of $S^2 \times S^1$ and a lens space.
Then $\mathcal{G}_{\{ D, E_1, E_2 \}} =
\langle \alpha \mid \alpha^2 \rangle \oplus \langle \tau \rangle$
and
$\mathcal{G}_{\{ D, E_1 \cup E_2 \}} =
\langle \alpha \mid \alpha^2 \rangle \oplus \langle \sigma \mid \sigma^2 \rangle
\oplus \langle \tau \rangle$.
\end{lemma}

\begin{proof}[Proof of Theorem $\ref{thm:presentations of the Goeritz groups}$]
(\ref{thm:lens and lens not reversible})
By Theorem \ref{thm:contractibility for lens and lens},
$\mathcal{SP}(V)$ is a tree.
By Lemmas \ref{lem:transitivity on reducing spheres}, 
the vertices modulo the action of $\mathcal{G}$ consists of
two classes, one contains $D$ and the other contains $E$.
Also, any edge of $\mathcal{SP}(V)$ is equal to
the edge $\{ D , E \}$ modulo the action of $\mathcal{G}$.
Therefore the quotient of $\mathcal{SP}(V)$ by the action
of $\mathcal{G}$ is an edge.
Now by the Bass-Serre theory and
Lemmas \ref{lem:stabilizer of a semi-primitive disk}
and \ref{lem:stabilizers of semi-primitive pairs}, we have
\begin{align*}
\mathcal{G} &=
\mathcal{G}_{\{D\}} *_{\mathcal{G}_{\{D,E\}}} \mathcal{G}_{\{E\}} \\
&= (\mathcal{G}_{\{ D, P \}} *_{\mathcal{G}_{\{D, P, Q_1 \}}}
\mathcal{G}_{\{D, P \cup Q_1\}} )
*_{\mathcal{G}_{\{D, P\}}}
(\mathcal{G}_{\{ E, P \}} *_{\mathcal{G}_{\{D, E, Q_2 \}}}
\mathcal{G}_{\{E, P \cup Q_2\}} )\\
&= (\langle \alpha \mid \alpha^2 \rangle \oplus
\langle \beta, \gamma_1 \mid {\gamma_1}^2 \rangle )
*_{\langle \alpha \mid \alpha^2 \rangle \oplus
\langle \beta \rangle}
(\langle \alpha \mid \alpha^2 \rangle \oplus
\langle \beta, \gamma_2 \mid {\gamma_2}^2 \rangle ) \\
&= \langle \alpha \mid \alpha^2 \rangle \oplus
\langle \beta, \gamma_1, \gamma_2
\mid {\gamma_1}^2, {\gamma_2}^2 \rangle .
\end{align*}

\noindent (\ref{thm:lens and lens reversible})
Again by Theorem \ref{thm:contractibility for lens and lens},
$\mathcal{SP}(V)$ is a tree.
Let $\mathcal{SP}'(V)$ be the first barycentric subdivision of
$\mathcal{SP}(V)$.
We note that the vertices of $\mathcal{SP}'(V)$ consists of
the vertices of $\mathcal{SP}(V)$ and
the barycenters of the edges of $\mathcal{SP}(V)$, each of which
corresponds to an unordered pair of vertices.
By Lemmas \ref{lem:transitivity on reducing spheres}, 
every vertex of $\mathcal{SP}'(V)$ is equal to the vertex
$D$ or the barycenter $\{D, E\}$, and
any edge of $\mathcal{SP}'(V)$ is equal to
the edge $\{D, \{ D , E \}\}$ modulo the action of $\mathcal{G}$.
Therefore the quotient of $\mathcal{SP}'(V)$ by the action
of $\mathcal{G}$ is an edge.
By the Bass-Serre theory and
Lemmas \ref{lem:stabilizer of a semi-primitive disk}
and \ref{lem:stabilizers of semi-primitive pairs}, we have
\begin{align*}
\mathcal{G} &=
\mathcal{G}_{\{D\}} *_{\mathcal{G}_{\{D,E\}}} \mathcal{G}_{\{D \cup E\}} \\
&= (\mathcal{G}_{\{ D, P \}} *_{\mathcal{G}_{\{D, P, Q_1 \}}}
\mathcal{G}_{\{D, P \cup Q_1\}} )
*_{\mathcal{G}_{\{D, P\}}}
(\mathcal{G}_{\{ D \cup E \}}  )\\
&= (\langle \alpha \mid \alpha^2 \rangle \oplus
\langle \beta, \gamma_1 \mid {\gamma_1}^2 \rangle )
*_{\langle \alpha \mid \alpha^2 \rangle \oplus
\langle \beta \rangle}
(\langle \alpha \mid \alpha^2 \rangle \oplus \langle \beta , \delta \mid \delta^2,
\delta \beta \delta = \alpha \beta \rangle) \\
&= \langle \alpha \mid \alpha^2 \rangle \oplus
\langle \beta, \gamma_1, \delta
\mid {\gamma_1}^2, \delta^2, \delta \beta \delta = \alpha \beta \rangle .
\end{align*}

\noindent (\ref{thm:lens and S2 S1})
We note that $\mathcal{G} = \mathcal{G}_{\{D\}}$.
Let $\mathcal{SP}_D'(V)$ be the first barycentric subdivision of
$\mathcal{SP}_D(V)$.
By Lemma \ref{lem:semi-primitive disks are interchangible},
the quotient of $\mathcal{SP}_D'(V)$ by the action
of $\mathcal{G}$ consists of an edge.
By the Bass-Serre theory and
Lemmas \ref{lem:stabilizer of a semi-primitive disk}
and \ref{lem:stabilizers of semi-primitive pairs}, we have
\begin{align*}
\mathcal{G}_{\{D\}}
&=
\mathcal{G}_{\{D, E_1\}} *_{\mathcal{G}_{\{D, E_1, E_2\}}}
\mathcal{G}_{\{D, E_1 \cup E_2\}} \\
&= (\langle \alpha \mid \alpha^2 \rangle \oplus \langle \beta, \gamma \mid \gamma^2 \rangle
\oplus \langle \tau \rangle)
*_{\langle \alpha \mid \alpha^2 \rangle \oplus \langle \tau \rangle}
(\langle \alpha \mid \alpha^2 \rangle \oplus \langle \sigma \mid \sigma^2 \rangle
\oplus \langle \tau \rangle) \\
&= \langle \alpha \mid \alpha^2 \rangle \oplus \langle \beta, \gamma, \sigma \mid \gamma^2, \sigma^2 \rangle
\oplus \langle \tau \rangle.
\end{align*}
This completes the proof.
\end{proof}

\noindent {\bf Acknowledgments.}
This work was carried out while the second-named author was visiting
Universit\`a di Pisa as a
JSPS Postdoctoral Fellow for Reserch Abroad.
He is grateful to the university and its staffs for
the warm hospitality.
The authors are grateful to the anonymous referee
for his or her very helpful comments that improved the presentation.


\begin{thebibliography}{99999}
\bibitem{Akb08}
Akbas, E.,
A presentation for the automorphisms of the 3-sphere that preserve a genus
two Heegaard splitting, Pacific J. Math. {\bf 236} (2008), 201--222.

\bibitem{Bir75}
Birman, J.,
On the equivalence of Heegaard splittings of closed, orientable 3 -manifolds.
{\it Knots, groups, and $3$-manifolds} ({\it Papers dedicated to the memory of R. H. Fox}),
pp. 137--164, Ann. of Math. Studies {\bf 84},
Princeton Univ. Press, Princeton, N.J., 1975.



\bibitem{BO83}
Bonahon, F., Otal, J.-P.,
Scindements de Heegaard des espaces lenticulaires, Ann. Sci. \'Ec. Norm. Sup. (4)
{\bf 16} (1983), 451--466.


\bibitem{Cho08}
Cho, S.,
Homeomorphisms of the 3-sphere that preserve a Heegaard splitting of genus
two, Proc. Amer. Math. Soc. {\bf 136} (2008), 1113--1123.

\bibitem{Cho12}
Cho, S.,
Genus two Goeritz groups of lens spaces, Pacific J. Math. {\bf 265} (2013), 1--16.

\bibitem{CK12}
Cho, S., Koda, Y.,
Primitive disk complexes for lens spaces, arXiv:1206.6243.

\bibitem{CK13a}
Cho, S., Koda, Y.,
The genus two Goeritz group of $S^2 \times S^1$, arXiv:1303.7145.

\bibitem{CK13b}
Cho, S., Koda, Y.,
Connected primitive disk complexes and Goeritz groups of lens spaces, in preparation.

\bibitem{CMS09}
Cho, S., McCullough, D., Seo, A.,
Arc distance equals level number,
Proc. Amer. Math. Soc. {\bf 137} (2009), 2801--2807.

\bibitem{Goe33}
Goeritz, L.,
Die Abbildungen der Brezelfl\"{a}ache und der Vollbrezel vom Geschlecht $2$.
Abh. Math. Semin. Hamb. Univ. {\bf 9} (1933), 244--259.

\bibitem{Go07}
Gordon, C.,
Problems. {\it Workshop on Heegaard Splittings}, pp. 401--411,
Geom. Topol. Monogr. {\bf 12}, Geom. Topol. Publ., Coventry, 2007.


\bibitem{Gra89}
Grasse, P.,
Finite presentation of mapping class groups of certain three-manifolds,
Topology Appl. {\bf 32} (1989) 295--305.

\bibitem{Hak68}
Haken, W.,
Some results on surfaces in 3-manifolds.
{\it Studies in Modern Topology}, pp. 39--98,
Math. Assoc. Amer., distributed by Prentice-Hall, Englewood Cliffs, N.J., 1968.


\bibitem{Hem01}
Hempel, J.,
3-manifolds as viewed from the curve complex,
Topology {\bf 40} (2001), 631--657.

\bibitem{Joh95}
Johannson, K.,
Topology and combinatorics of 3-manifolds,
Lecture Notes in Mathematics {\bf 1599}, Springer-Verlag, Berlin, 1995.

\bibitem{Joh10}
Johnson, J.,
Mapping class groups of medium distance Heegaard splittings,
Proc. Amer. Math. Soc. {\bf 138} (2010), 4529--4535.

\bibitem{Joh11}
Johnson, J.,
Heegaard splittings and open books,
arXiv:1108.5302.

\bibitem{Joh12}
Johnson, J.,
Mapping class groups of Heegaard splittings of surface bundles,
arXiv:1201.2628.

\bibitem{Kod11}
Koda, Y.,
Automorphisms of the 3-sphere that preserve spatial graphs and handlebody-knots,
arXiv:1106.4777.

\bibitem{Lei05}
Lei, F.,
Haken spheres in the connected sum of two lens spaces,
Math. Proc. Camb. Philos. Soc. {\bf 138} (2005), 97--105.

\bibitem{LZ04}
Lei, F., Zhang, Y.,
Haken spheres in genus 2 Heegaard splittings
of nonprime 3-manifolds,
Topology Appl. {\bf 142} (2004) 101--111.

\bibitem{McC91}
McCullough, D.,
Virtually geometrically finite mapping class groups of $3$-manifolds,
J. Diff. Geom. {\bf 33} (1991), 1--65.

\bibitem{MS88}
Montesinos, J., Safont, C.,
On the Birman invariants of Heegaard splittings,
Pacific J. Math. {\bf 132} (1988), 113--142.

\bibitem{Nam07}
Namazi, H.,
Big Heegaard distance implies finite mapping class group, Topology Appl. {\bf 154} (2007), 2939--2949.


\bibitem{OZ81}
Osborne, R. P., Zieschang, H.,
Primitives in the free group on two generators,
Invent. Math. {\bf 63} (1981), no. 1, 17--24.

\bibitem{Sch04}
Scharlemann, M.,
Automorphisms of the $3$-sphere that preserve a
genus two Heegaard splitting, Bol. Soc. Mat. Mexicana {\bf 10} (2004),
503--514.

\bibitem{Seo08} Seo, A.,
Torus leveling of $(1,1)$-knots,
dissertation at the University of Oklahoma (2008).


\bibitem{Ser77} Serre, J.-P.,
\emph{Arbres, amalgames, $\mathrm{SL_2}$},
Ast\'{e}risque {\bf 46}, Soci\'{e}t\'{e} Math\'{e}matique de France, Paris, 1977.



\bibitem{Suz77}
Suzuki, S.,
On homeomorphisms of a 3-dimensional handlebody, Can. J. Math.
{\bf 29} (1977), 111--124.

\bibitem{Waj98}
Wajnryb, B.,
Mapping class group of a handlebody,
Fund. Math. {\bf 158} (1998), 195--228.



\end{thebibliography}
\end{document}